\documentclass[10pt,reqno]{amsart}
\usepackage{latexsym,amsmath}
\usepackage{enumerate}
\usepackage{amsfonts}
\usepackage{amssymb}
\usepackage{latexsym}

\usepackage[T1]{fontenc}
\usepackage{fourier}
\usepackage{bbm}
\usepackage{verbatim}

\usepackage{color}
\usepackage{fullpage}

\newcommand\nix{\,\cdot\,}
\newcommand\vV{\vec V}

\newcommand\vS{\vec S}

\newcommand\reg{\G}

\newcommand\G{\vec G}

\newcommand\vH{\vec H}

\numberwithin{equation}{section}

\renewcommand{\vec}[1]{\boldsymbol{#1}}

\newcommand\SIGMA{\vec\sigma}
\newcommand\TAU{\vec\tau}

\newtheorem{definition}{Definition}[section]
\newtheorem{claim}[definition]{Claim}

\newtheorem{theorem}[definition]{Theorem}
\newtheorem{lemma}[definition]{Lemma}

\newtheorem{corollary}[definition]{Corollary}

\newtheorem{fact}[definition]{Fact}

\newcommand\cA{\mathcal{A}}
\newcommand\cB{\mathcal{B}}

\newcommand\cD{\mathcal{D}}

\newcommand\cE{\mathcal{E}}

\newcommand\cS{\mathcal{S}}

\newcommand\cP{\mathcal{P}}
\newcommand\cX{\mathcal{X}}

\def\cE{{\mathcal E}}

\newcommand\eps{\varepsilon}
\newcommand\del{\delta}

\newcommand\Erw{\mathrm{E}}
\newcommand{\vecone}{\vec{1}}

\newcommand{\Po}{{\rm Po}}
\newcommand{\Bin}{{\rm Bin}}

\newcommand{\Be}{{\rm Be}}

\newcommand\TV[1]{\left\|{#1}\right\|_{\mathrm{TV}}}

\newcommand\bc[1]{\left({#1}\right)}

\newcommand\bcfr[2]{\bc{\frac{#1}{#2}}}
\newcommand{\bck}[1]{\left\langle{#1}\right\rangle}
\newcommand\brk[1]{\left\lbrack{#1}\right\rbrack}

\newcommand\abs[1]{\left|{#1}\right|}

\newcommand\RR{\mathbb{R}}

\newcommand{\Whp}{A.a.s.}
\newcommand{\whp}{a.a.s.}

\newcommand{\stacksign}[2]{{\stackrel{\mbox{\scriptsize #1}}{#2}}}
\newcommand{\tensor}{\otimes}

\newcommand{\Erdos}{Erd\H{o}s}
\newcommand{\Renyi}{R\'enyi}
\newcommand{\Lovasz}{Lov\'asz}

\newcommand{\Szemeredi}{Szemer\'edi}

\newcommand\pr{\mathrm{P}} 

\newcommand\Lem{Lemma}
\newcommand\Prop{Proposition}
\newcommand\Thm{Theorem}

\newcommand\Cor{Corollary}
\newcommand\Sec{Section}
\newcommand\Chap{Chapter}

\begin{document}

\title{Belief Propagation on replica symmetric random factor graph models}

\author[Coja-Oghlan and Perkins]{Amin Coja-Oghlan$^{*}$, Will Perkins}
\thanks{$^{*}$The research leading to these results has received funding from the European Research Council under the European Union's Seventh 
Framework Programme (FP7/2007-2013) / ERC Grant Agreement n.\ 278857--PTCC}

\address{Amin Coja-Oghlan, {\tt acoghlan@math.uni-frankfurt.de}, Goethe University, Mathematics Institute, 10 Robert Mayer St, Frankfurt 60325, Germany.}

\address{Will Perkins, {\tt math@willperkins.org}, School of Mathematics, University of Birmingham, Edgbaston, Birmingham, UK.}

\begin{abstract}
\noindent
According to physics predictions, the free energy of random factor graph models that satisfy a certain ``static replica symmetry'' condition
can be calculated via the Belief Propagation message passing scheme [Krzakala et al., PNAS 2007].
Here we prove this conjecture for two general classes of random factor graph models, namely Poisson random factor graphs and random regular factor graphs.
Specifically, we show that the messages constructed just as in the case of acyclic factor graphs asymptotically satisfy the Belief Propagation equations
and that the free energy density is given by the Bethe free energy formula.

\end{abstract}

\maketitle

\section{Introduction}\label{Sec_intro}

\subsection{Belief Propagation}
Factor graph models are ubiquitous in statistical physics, computer science and combinatorics~\cite{kschischang2001factor,Rudi}.
Formally, a {\em factor graph} $G=(V(G),F(G),\partial_G,(\psi_a)_{a\in F(G)})$ consists of a finite set $V(G)$ of {\em variable nodes},
a set $F(G)$ of {\em constraint nodes} and a function $\partial_G:F(G)\to\bigcup_{l\geq 0}V(G)^l$
	that assigns each constraint node $a\in F(G)$ a finite sequence $\partial a=\partial_Ga$ of variable nodes, whose length is denoted by $d(a)=d_G(a)$.
Additionally, there is a finite set $\Omega$ of {\em spins} and each constraint node $a\in F$ comes with a {\em weight function}
	$\psi_a:\Omega^{d(a)}\to(0,\infty)$.
The factor graph gives rise to a probability distribution $\mu_G$, the {\em Gibbs measure}, on the set $\Omega^{V(G)}$.
Indeed, letting $\sigma(x_1,\ldots,x_k)=(\sigma(x_1),\ldots,\sigma(x_k))$ for $\sigma\in\Omega^{V(G)}$ and $x_1,\ldots,x_k\in V(G)$, we define
	\begin{align}\label{eqGibbs}
	\mu_G&:\sigma\in\Omega^{V(G)}\mapsto\frac1{Z_G}\prod_{a\in F(G)}\psi_a(\sigma(\partial a)),\qquad\mbox{where}
		\quad Z_G=\sum_{\tau\in\Omega^{V(G)}}\prod_{a\in F(G)}\psi_a(\sigma(\partial a))
	\end{align}
is the {\em partition function}.
Moreover, $G$ induces a bipartite graph on $V(G)\cup F(G)$ in which the constraint node $a$
is adjacent to the variable nodes that appear in the sequence $\partial a$.
By (slight) abuse of notation we just write $\partial a=\partial_Ga$ for the set of such variable nodes. 
Conversely, for $x\in V(G)$ we let $\partial x=\partial_Gx$ be the set of all $a\in F(G)$ such that $x\in\partial a$
and we let $d(x)=d_G(x)=|\partial x|$.
(However, we keep in mind that the order of the neighbors of $a$ matters, unless the weight function $\psi_a$ is permutation invariant.)

The {\em Potts model} on a finite lattice is an example of a factor graph model.
In this case the lattice points correspond to the variable nodes and each edge $\{x,y\}$ of the lattice gives rise to a constraint node $a$.
The spins are $\Omega=\{1,\ldots,q\}$ for some integer $q\geq2$.
Moreover, all constraint nodes have the same weight function, namely $\Omega^2\to(0,\infty)$, $(s,t)\mapsto\exp(\beta\vecone\{s=t\})$, 
where $\beta$ is a real parameter.

Another example is the {\em $k$-SAT model} for some $k\geq2$.
The variable nodes $x_1,\ldots,x_n$ correspond to Boolean variables and the constraint nodes $a_1,\ldots,a_m$ to $k$-clauses.
The set of possible spins is $\Omega=\{\pm1\}$ and each constraint node comes with a $k$-tuple $s_i=(s_{i1},\ldots,s_{ik})\in\{\pm1\}^k$.
The weight function is
	$\psi_{a_i}:\{\pm1\}^k\to(0,\infty)$, $\sigma\mapsto\exp(-\beta\vecone\{\sigma=s_i\})$,
	where $\beta>0$ is a real parameter.
Combinatorially, $\pm1$ represent the Boolean values `true' and `false'
and $a_i$ is a propositional clause on the variables $\partial a_i$ whose $j$th variable is negated iff $s_{ij}=-1$.

A key problem associated with a factor graph model is to 	analytically or algorithmically calculate  the ``free energy'' $\ln Z_G$.
Either way, this is notoriously difficult in general~\cite{MooreMertens}.
But in the (very) special case that $G$, viz.\ the associated bipartite graph, is acyclic it is well known that this problem
can be solved via the {\em Belief Propagation equations} (see eg.~\cite[ch.~14]{MM}).
More precisely, for a variable node $x$ and a constraint node $a$ such that $x\in\partial a$ let
$\mu_{G,x\to a}$ be the marginal of $x$ with respect to the Gibbs measure of the factor graph $G-a$ obtained from $G$ by deleting the constraint node $a$.
(To be explicit, $\mu_{G,x\to a}(\sigma)$ is the probability that $x$ is assigned the spin $\sigma\in\Omega$ in a random configuration
	$\SIGMA\in\Omega^{V(G)}$ drawn from $\mu_{G-a}$.)
Similarly, let $\mu_{G,a\to x}$ be the marginal of $x$ in the factor graph obtained from
$G$ by deleting all constraint nodes $b\in\partial x\setminus a$.
We call $\mu_{G,x\to a}$ the {\em message} from $x$ to $a$ and conversely $\mu_{G,a\to x}$ the message from $a$ to $x$.
If $G$ is acyclic, then for all $x\in V(G),\ a\in\partial x,\ \sigma\in\Omega$,
	\begin{align}\label{eqBP}
	\mu_{G,x\to a}(\sigma)&=\frac{\prod_{b\in\partial x}\mu_{G,b\to x}(\sigma)}{\sum_{\tau\in\Omega}
		\prod_{b\in\partial x}\mu_{G,b\to x}(\tau)},&
	\mu_{G,a\to x}(\sigma)&=\frac{\sum_{\tau\in\Omega^{\partial a}}\vecone\{\tau(x)=\sigma\}\psi_a(\tau)
		\prod_{y\in\partial a\setminus x}\mu_{G,y\to a}(\tau(y))}{
		\sum_{\tau\in\Omega^{\partial a}}\psi_a(\tau)\prod_{y\in\partial a\setminus x}\mu_{G,y\to a}(\tau(y))
		}
	\end{align}
and the messages $\mu_{G,x\to a},\mu_{G,a\to x}$ are the  unique solution to (\ref{eqBP}).
In fact, the messages can be computed via a fixed point iteration and the number of iterations steps required is bounded by the diameter of $G$.
Furthermore, $\ln Z_G$ is equal to the {\em Bethe free energy}, defined in terms of the messages as%
	\begin{align*}
	\cB_G&=
	\sum_{x\in V(G)}\ln\brk{\sum_{\tau\in\Omega}\prod_{b\in\partial x}\mu_{G,b\to x}(\tau)}+
		\sum_{a\in F(G)}\ln\brk{\sum_{\tau\in\Omega^{\partial a}}\psi_a(\tau)\prod_{x\in\partial a}\mu_{G,x\to a}(\tau(x))}
		-\hspace{-1mm}\sum_{\substack{a\in F(G)\\x\in\partial a}}\hspace{-1mm}\ln\brk{\sum_{\sigma\in\Omega}\mu_{G,a\to x}(\sigma)\mu_{G,x\to a}(\sigma)}.
	\end{align*}
(The denominators in (\ref{eqBP}) and the arguments of the logarithms in the Bethe free energy are guaranteed to be positive
	because we assume that the weight functions $\psi_a$ take strictly positive values.)

\subsection{Random factor graphs}\label{Sec_IntroPoisson}
The present paper is about Gibbs distributions arising from random models of factor graphs.
Such models are of substantial interest in combinatorics, computer science and information theory~\cite{ANP,Rudi}.
The following setup encompasses a reasonably wide class of models.
Let $\Omega$ be a finite set of `spins', let $k\geq3$ be an integer, let $\Psi\neq\emptyset$ be a finite set of functions
$\psi:\Omega^k\to(0,\infty)$ and let $\rho=(\rho_\psi)_{\psi\in\Psi}$ be a probability distribution on $\Psi$.
Then for an integer $n>0$ and a real $d>0$ we define the ``Poisson'' random factor graph $\G_n=\G_n(d,\Omega,k,\Psi,\rho)$ as follows.
The set of variable nodes is $V(\G_n)=\{x_1,\ldots,x_n\}$
and the set of constraint nodes is $F(\G_n)=\{a_1,\ldots,a_m\}$, where $m$ is a Poisson random variable with mean $d n/k$.
Furthermore, independently for each $i=1,\ldots,m$ a weight function $\psi_{a_i}\in\Psi$ is chosen from the distribution $\rho$.
Finally, $\partial a_i\in\{x_1,\ldots,x_n\}^k$ is a uniformly random $k$-tuple of variables, chosen independently for each $i$.
For fixed $d,\Omega,k,\Psi,\rho$
the random factor graph $\G_n$ has a property $\cA$ {\em asymptotically almost surely} (`\whp') if $\lim_{n\to\infty}\pr\brk{\G_n\in\cA}=1$.

A well known concrete example is the {\em random $k$-SAT model} for $k\geq2$, where
we let $\Omega=\{\pm1\}$ and $\Psi=\{\psi^{(s)}:s\in\{\pm1\}^k\}$ with
 $\psi^{(s)}:\sigma\in\{\pm1\}^k\mapsto\exp(-\beta\vecone\{\sigma=s\})$ and $\rho$ is the uniform distribution on $\Psi$.
Further prominent examples include the Ising and the Potts models on the \Erdos-\Renyi\ random graph~\cite{DM,demboPotts}.

As in the general case, it is a fundamental challenge is to get a handle on the free energy $\ln Z_{\G_n}$.
To this end, physicists have proposed the ingenious albeit non-rigorous ``cavity method''~\cite{mezard1990spin}.
The simplest version of this approach, the {\em replica symmetric ansatz}, basically treats  the random factor graph as though it were acyclic.
In particular, the replica symmetric ansatz holds that the ``messages'' $\mu_{\G_n,x\to a}$, $\mu_{\G_n,a\to x}$, defined just as in the tree case as the marginals
of the factor graph obtained by removing $a$ resp.\ $\partial x\setminus a$, satisfy the 
Belief Propagation equations (\ref{eqBP}), at least asymptotically as $n\to\infty$.
Moreover, the replica symmetric prediction as to the free energy is nothing but the Bethe free energy $\cB_{\G_n}$.
If so, then Belief Propagation can not just be used as an analytic tool, but potentially also as an efficient ``message passing algorithm''~\cite{pnas}.
Indeed, the Belief Propagation fixed point iteration has been used algorithmically with considerable empirical success~\cite{Kroc}.

Under what assumptions can we vindicate the replica symmetric ansatz?
Let us write $\mu_{G,x}$ for the marginal of a variable node $x$ under $\mu_{G}$.
Moreover, write $\mu_{G,x,y}$ for the joint distribution of two variable nodes $x,y$ and let $\TV\nix$ denote the total variation norm.
Then
	\begin{align}\label{eqRS}
	\lim_{n\to\infty}\frac1{n^2}\sum_{i,j=1}^n\Erw\TV{\mu_{\G_n,x_i,x_j}-\mu_{\G_n,x_i}\tensor\mu_{\G_n,x_j}}=0
	\end{align}
expresses that \whp\ the spins of two randomly chosen variable nodes are asymptotically independent.
An important conjecture holds that (\ref{eqRS}) is sufficient for the success of Belief Propagation and the Bethe formula~\cite{pnas}.

The main result of this paper proves this conjecture.
For a given factor graph $G$ 
we call the family of messages $\mu_{G,\nix\to\nix}=(\mu_{G,x\to a},\mu_{G,a\to x})_{x\in V(G),a\in F(G),x\in\partial a}$ an {\em $\eps$-Belief Propagation fixed point} on $G$ if
	\begin{align*}
	\sum_{\substack{x\in V(G)\\a\in\partial x\\\sigma\in\Omega}}\abs{\mu_{G,x\to a}(\sigma)-
			\frac{\prod_{b\in\partial x\setminus a}\mu_{G,b\to x}(\sigma)}
				{\sum_{\tau\in\Omega}\prod_{b\in\partial x\setminus a}\mu_{G,b\to x}(\tau)}}
				+
				\abs{\mu_{G,a\to x}(\sigma)-
			\frac{\sum_{\tau\in\Omega^{\partial a}}\vecone\{\tau(x)=\sigma\}\psi_a(\tau)\prod_{y\in\partial a\setminus x}\mu_{G,y\to a}(\tau(y))}
				{\sum_{\tau\in\Omega^{\partial a}}\psi_a(\tau)\prod_{y\in\partial a\setminus x}\mu_{G,y\to a}(\tau(y))}}&<\eps n.
	\end{align*}
Thus, the equations (\ref{eqBP}) hold approximately for almost all pairs $x\in V(G)$, $a\in\partial x$.

\begin{theorem}\label{Thm_RSBP}
If (\ref{eqRS}) holds, then there is a sequence $(\eps_n)_n\to0$ such that $\mu_{\G_n,\nix\to\nix}$ is an $\eps_n$-Belief Propagation fixed point \whp
\end{theorem}

\begin{corollary}\label{Thm_RSBethe}
If (\ref{eqRS}) holds and $\frac1n\cB_{\G_n}$ converges to a real number $B$ in probability, then $\lim_{n\to\infty}\frac1n\Erw[\ln Z_{\G}]=B.$
\end{corollary}

If (\ref{eqBP}) holds exactly, then the Bethe free energy can be rewritten in terms of the marginals of the variable and constraint nodes~\cite{YFW}.
Specifically, write $\mu_{G,a}$ for the joint distribution of the variables $\partial a$ and let
	\begin{align*}
	\cB_G'&=\sum_{x\in V(G)}(d_G(x)-1)\sum_{\sigma\in\Omega}\mu_{G,x}(\sigma)\ln\mu_{G,x}(\sigma)+
		\sum_{a\in F(G)}\sum_{\sigma\in\Omega^{\partial a}}\mu_{G,a}(\sigma)(\ln\psi_a(\sigma)-\ln\mu_{G,a}(\sigma)).
	\end{align*}
Once more the fact that all $\psi\in\Psi$ are strictly positive ensures that $\cB_G'$ is well-defined.

\begin{corollary}\label{Cor_RSBethe}
If (\ref{eqRS}) holds and $\frac1n\cB_{\G_n}'$ converges to a real $B'$ in probability, then $\lim_{n\to\infty}\frac1n\Erw[\ln Z_{\G}]=B'.$
\end{corollary}

\subsection{Random regular models}
In a second important class of random factor graph models all variable nodes have the same degree $d$.
Thus, with $\Omega,k,\Psi,\rho$ as before let $\reg_n=\reg_{n,\mathrm{reg}}(d,\Omega,k,\Psi,\rho)$ be the random factor graph
with variable nodes $x_1,\ldots,x_n$ and constraint nodes $a_1,\ldots,a_m$, $m=\lfloor dn/k\rfloor$,
chosen uniformly from the set of all 
factor graphs $G$ with $d_G(x_i)\leq d$ for all $i$.
As before, the weight functions $\psi_{a_i}\in\Psi$ are chosen independently from  $\rho$.
Clearly, if $k$ divides $dn$, then all variable nodes have degree $d$ exactly.

In order to study $\reg_n$ we introduce a ``percolated'' version of this model.
For $\psi:\Omega^k\to(0,\infty)$ and $J\subset[k]$ let
	$$\textstyle\psi^J:\Omega^J\to(0,\infty),\qquad
		(\sigma_j)_{j\in J}\mapsto\Omega^{|J|-k}\sum_{(\sigma_j)_{j\not\in J}\in\Omega^{k-|J|}}\psi(\sigma).$$
In words, $\psi^J$ is obtained from $\psi$ by averaging over the missing coordinates $j\in\{1,\ldots,k\}\setminus J$; thus, $\psi^{\{1,\ldots,k\}}=\psi$.
Further, given $\eps>0$ let $\reg^\eps_{n}=\reg_{n,\mathrm{reg}}^\eps(d,\Omega,k,\Psi,\rho)$ be the random factor graph with
variable nodes $x_1,\ldots,x_n$ obtained via the following experiment.
\begin{description}
\item[REG1] Choose a random number $m=\Po(dn/k)$.
\item[REG2] Independently for each $i\in\{1,\ldots,m\}$,
	\begin{enumerate}[(a)]
	\item obtain $J_i\subset\{1,\ldots,k\}$ by including each number with probability $1-\eps$ independently and
	\item choose a function $\psi_i\in\Psi$ according to $\rho$ and let $\psi_{a_i}=\psi_i^{J_i}$.
	\end{enumerate}
\item[REG3] If $\sum_{i=1}^m|J_i|>dn$, then start over from {\bf REG1}.
	Otherwise choose $\G_n^\eps$ uniformly at random subject to the condition that no variable node has degree greater than $d$.  
\end{description}

A practical method to sample  $\G_n^\eps$ uniformly at random is via the ``configuration model''~\cite[\Chap~9]{JLR}: we create $d$ `clones' of each variable node and $|J_i|$ clones of each constraint node $a_i$ (keeping the clones ordered), then pick a uniformly random maximum matching between variable node clones and constraint node clones, then collapse the matching to give our random factor graph ($x$ attached to constraint $a_i$ if some clone of $x$ is matched with a clone of $a_i$). Note that $\sum_{i=1}^m|J_i|$ has the distribution $\Bin(X,1-\eps)$ where $X$ has distribution $k$ times a $\Po(dn/k)$. In particular, its mean is $(1-\eps) dn$ and so a Chernoff bound gives 
\begin{align}
\label{eq:chernoff}
\Pr[Y > (1-\eps/2) dn ] \le \exp(-\Omega(\eps^2 n)).
\end{align}
 
\begin{theorem}\label{Thm_regRSBP}
Assume that $\eps >0$ is such that
	\begin{align}\label{eqregRS}
	\lim_{n\to\infty}\frac1{n^2}\sum_{i,j=1}^n\Erw\TV{\mu_{\reg^\eps_n,x_i,x_j}-\mu_{\reg^\eps_n,x_i}\tensor\mu_{\reg^\eps_n,x_j}}=0.
	\end{align}
Then there is $(\delta_n)_n\to0$ such that $\mu_{\G_n^\eps,\nix\to\nix}$ is a $\delta_n$-Belief Propagation fixed point \whp
\end{theorem}

\noindent
Indeed, if (\ref{eqregRS}) holds for all small enough $\eps>0$, then we obtain
	the free energy of $\reg_n=\reg_{n,\mathrm{reg}}(d,\Omega,k,\Psi,\rho)$.

\begin{corollary}\label{Thm_regRSBethe}
Assume that there is some $\eps_0>0$ such that  (\ref{eqregRS}) holds for every $\eps\in(0,\eps_0)$ and that there is $B\in\RR$ such that
	$\lim_{\eps\searrow0}\limsup_{n\to\infty}\Erw\abs{n^{-1}\cB_{\reg_n^\eps}-B}=0.$
Then $\lim_{n\to\infty}\frac1n\Erw[\ln Z_{\reg_n}]=B.$
\end{corollary}

\subsection{Non-reconstruction}
In physics jargon factor graph models that satisfy (\ref{eqRS}) resp.~(\ref{eqregRS}) are called {\em statically replica symmetric}.
An obvious question is how (\ref{eqRS}) and (\ref{eqregRS}) can be established ``in practice''.
One simple sufficient condition is the more geometric notion of non-reconstruction, also known as  {\em dynamic replica symmetry} in physics.
To state it, recall the bipartite graph on the set of variable and constraint nodes that a factor graph induces.
This bipartite graph gives rise to a metric on the set of variable and constraint nodes, namely the length of a shortest path.
Now, for a factor graph $G$, a variable node $x$, an integer $\ell\geq1$ and a configuration $\sigma\in\Omega^{V(G)}$
we let $\nabla_\ell(G,x,\sigma)$ be the set of all $\tau\in\Omega^{V(G)}$ such that $\tau(y)=\sigma(y)$
for all $y\in V(G)$ whose distance from $x$ exceeds $\ell$.
The random factor graph 
$\G_n=\G_{n}(d,\Omega,k,\Psi,\rho)$ or $\G_n=\reg_{n,\mathrm{reg}}^\eps(d,\Omega,k,\Psi,\rho)$ has the {\em non-reconstruction property} if
	\begin{align}\label{eqNonReconstruction}
	\lim_{\ell\to\infty}\limsup_{n\to\infty}
		\frac{1}{n}\sum_{i=1}^n\sum_{\sigma\in\Omega^n}\Erw\brk{\mu_{\G_n}(\sigma)
			\TV{\mu_{\G_n,x_i}-\mu_{\G_n,x_i}[\nix|\nabla_\ell(\G_n,x_i,\sigma)]}}&=0.
	\end{align}
where the expectation is over the choice of $\G_n$. 
In words, for large enough $\ell$ and $n$ the random factor graph $\G_n$ has the following property \whp\
If we pick a variable node $x_i$ uniformly at random and if we pick $\SIGMA$ randomly from the Gibbs distribution,
then the expected difference between the ``pure'' marginal $\mu_{\G_n,x_i}$ of $x_i$ and the marginal of $x_i$
in the conditional distribution given that the event $\nabla_\ell(\G_n,x_i,\SIGMA)$ occurs diminishes.
We contrast~(\ref{eqNonReconstruction}) to the much stronger {\em uniqueness} property which states that the influence of the worst-case boundary condition on the marginal spin distribution of $x_i$ decreases in the limit of large $\ell$ and $n$.

\begin{lemma}\label{prop:nonreconstruction}
Let $\G_n$ be distributed according to $\G_{n}(d,\Omega,k,\Psi,\rho)$ or $\reg_{n,\mathrm{reg}}^\eps(d,\Omega,k,\Psi,\rho)$.
If (\ref{eqNonReconstruction}) holds, then 
\begin{equation}
\label{eq:rsnonrecon}
\lim_{n\to\infty} \frac{1}{n^2} \sum_{i,j=1}^n \Erw\TV{\mu_{\G_n,x_i,x_j}-\mu_{\G_n,x_i}\tensor\mu_{\G_n,x_j}}=0. 
\end{equation}
\end{lemma}

\noindent
Non-reconstruction is a sufficient but not a necessary condition for~(\ref{eqRS}) and~(\ref{eqregRS}).
For instance, in the random graph coloring problem (\ref{eqRS}) is satisfied in a {\em much} wider regime of parameters than (\ref{eqNonReconstruction})
	\cite{Nor,pnas,montanari2011reconstruction}.

\subsection{Discussion and related work}
The main results of the present paper match the predictions from~\cite{pnas} and thus provide a fairly comprehensive vindication of Belief Propagation.
To the extent that Belief Propagation and the Bethe free energy are not expected to be correct if the conditions (\ref{eqRS}) resp.\ (\ref{eqregRS}) 
are violated~\cite{pnas,MM}, 
the present results seem to be best possible.

In combination with \Lem~\ref{prop:nonreconstruction} the main results facilitate the ``practical'' use of Belief Propagation to analyze the free energy.
For instance, \Thm~\ref{Thm_regRSBP} and \Cor~\ref{Thm_regRSBethe} allow for a substantially simpler derivation
of  the condensation phase transition in the regular $k$-SAT model than in the original paper~\cite{Victor2}.
Although non-trivial it is practically feasible to study Belief Propagation fixed points on random factor graphs; e.g., \cite{Victor2,Cond}.

Additionally, as \Thm s~\ref{Thm_RSBP} and~\ref{Thm_regRSBP}  show that the ``correct'' messages 
are an asymptotic Belief Propagation fixed point,  these results probably go as far as one can hope for in terms
of a generic explanation of the algorithmic success of Belief Propagation.
The missing piece in order to actually prove that the Belief Propagation fixed point iteration converges rapidly is basically an analysis of the ``basin of attraction''.
However, this will likely have to depend on the specific model.

We always assume that the weight functions $\psi_a$ associated with the constraint nodes are strictly positive.
But this is partly out of convenience (to ensure that all the quantities that we work with are well-defined, no questions asked).
For instance, it is straightforward to extend the present arguments extend to the hard-core model on independent sets (details omitted).

In an important paper, Dembo and Montanari~\cite{DM} made progress towards putting the physics predictions on factor graphs, random or not, on a rigorous basis.
They proved, inter alia, that a certain ``long-range correlation decay'' property
reminiscent of non-reconstruction
 is sufficient for the Belief Propagation equations to hold on a certain class of factor graphs whose local neighborhoods converge to trees~\cite[\Thm~3.14]{DM}.
Following this, under the assumption of Gibbs uniqueness along an interpolating path in parameter space, Dembo, Montanari, and Sun~\cite{dembo} verified the Bethe free energy formula for locally tree-like factor graphs with a single weight function and constraint nodes of degree $2$.
Based on these ideas Dembo, Montanari, Sly and Sun~\cite{demboPotts} verified the Bethe free energy prediction for the ferromagnetic Potts model on regular tree-like graphs at any temperature.

The present paper builds upon the ``regularity lemma'' for measures on discrete cubes from~\cite{Victor}.
In combinatorics, the ``regularity method'', which developed out of \Szemeredi's regularity lemma for graphs~\cite{Szemeredi}, has become an indispensable tool.
Bapst and Coja-Oghlan~\cite{Victor} adapted \Szemeredi's proof to measures on a discrete cube, such as the Gibbs measure of a (random) factor graph,
and showed that this result can be combined with the ``second moment method'' to calculate the free energy under certain assumptions.
While these assumptions are (far) more restrictive than our conditions (\ref{eqRS}) and~(\ref{eqregRS}), \cite{Victor} deals with more general factor graph models.

Furthermore, inspired by the theory of graph limits~\cite{Lovasz},
Coja-Oghlan, Perkins and Skubch~\cite{COPS} put forward a ``limiting theory'' for discrete probability measures to go with the regularity concept from~\cite{Victor}.
They applied this concept to the Poisson factor graph model from \Sec~\ref{Sec_IntroPoisson}
under the assumption that (\ref{eqRS}) holds {\em and} that the Gibbs measure converges in probability to a limiting measure (in the topology constructed in~\cite{COPS}).
While these assumptions are stronger and more complicated to state than (\ref{eqRS}),
\cite{COPS} shows that the limiting Gibbs measure induces a ``geometric'' Gibbs measure on a certain infinite random tree.
Moreover, this geometric measure satisfies a certain fixed point relation reminiscent of the Belief Propagation equations.

Additionally, the present paper builds upon ideas from Panchenko's work~\cite{Panchenko,Panchenko2,PanchenkoBook}.
In particular, we follow~\cite{Panchenko,Panchenko2,PanchenkoBook} in using the Aizenman-Sims-Starr scheme~\cite{Aizenman} to calculate the free energy.
Moreover, although Panchenko only deals with Poisson factor graphs,
the idea of percolating the regular factor graph is inspired by his ``cavity coordinates''
as well as the interpolation argument of Bayati, Gamarnik and Tetali~\cite{bayati}.  Other applications of the cavity method to computing the free energy of Gibbs distributions on lattices include \cite{gamarnikKatz}.

The paper~\cite{Panchenko2} provides a promising approach towards a general formula for the free energy in Poisson random factor graph models.
Specifically, \cite{Panchenko2} yields a variational formula for the free energy under the assumption that the
Gibbs measures satisfies a ``finite replica symmetry breaking'' condition, which is more general than (\ref{eqRS}).
Another assumption of~\cite{Panchenko2} is that the weight functions of the factor graph model must satisfy certain ``convexity conditions'' to facilitate the use of the interpolation method, which is needed to upper-bound the free energy.
However, it is conceivable that the interpolation argument is not necessary if (\ref{eqRS}) holds and that \Cor~\ref{Cor_RSBethe}
could be derived along the lines of~\cite{Panchenko2} (although this is not mention in the paper).
In any case, the main point of the present paper is to justify the Belief Propagation equations, which are at very core of the physicists ``cavity method'' in factor graph models,
and to obtain a formula for the free energy in terms of ``messages''.

Finally, the proof of \Lem~\ref{prop:nonreconstruction} is a fairly straightforward extension of the proof of~\cite[\Prop~3.4]{COPS}.
That proof, in turn, is a generalization of an argument from~\cite{Mossel}. For more on non-reconstruction thresholds in random factor graph models see \cite{Nayantara,Charis,GM,montanari2011reconstruction}.

\subsection{Outline}
After introducing some notation and summarizing the results from~\cite{Victor} that we build upon in \Sec~\ref{Sec_prelims},
we prove \Thm~\ref{Thm_RSBP} and Corollaries~\ref{Thm_RSBethe} and~\ref{Cor_RSBethe} in \Sec~\ref{Sec_Poisson}.
\Sec~\ref{Sec_regular} then deals with \Thm~\ref{Thm_regRSBP} and \Cor~\ref{Thm_regRSBethe}.
Finally, the short proof of \Lem~\ref{prop:nonreconstruction} can be found in \Sec~\ref{Sec_nonre}.

\section{Preliminaries}\label{Sec_prelims}

\noindent
For an integer $l\geq1$ we let $[l]=\{1,\ldots,l\}$.
When using $O(\nix)$-notation we refer to the asymptotics as $n\to\infty$ by default.
We say two sequences of probability distributions $Q_n$ and $P_n$ are {\em mutually contiguous} if for every sequence of events $E_n$, $P_n(E_n) =o(1)$ if and only if $Q_n(E_n) = o(1)$. Throughout the paper we denote by $d,\Omega,k,\Psi,\rho$ the parameters of the factor graph models from \Sec~\ref{Sec_intro}.
We always assume $d,\Omega,k,\Psi,\rho$ remain fixed as $n\to\infty$.

For a finite set $\cX$ we let $\cP(\cX)$ be the set of all probability measures on $\cX$,
which we identify with the set of all maps $p:\cX\to[0,1]$ such that $\sum_{\omega\in\cX}p(\omega)=1$.
If $\mu\in\cP(\cX^S)$ for some finite set $S\neq\emptyset$, then we write
	$\TAU^\mu,\SIGMA^\mu,\SIGMA_1^\mu,\SIGMA_2^\mu,\ldots$ for independent samples chosen from $\mu$.
We omit the superscript where possible.
Furthermore, if $X:(\cX^S)^l\to\RR$ is a random variable, then we write
	\begin{align*}
	\bck{X}_\mu=\bck{X(\SIGMA^\mu_1,\ldots,\SIGMA^\mu_l)}_\mu=\sum_{\sigma_1,\ldots,\sigma_l\in\cX^S}
		X(\sigma_1,\ldots,\sigma_l)\prod_{i=1}^l\mu(\sigma_i)
	\end{align*}
for the expectation of $X$ with respect to $\mu^{\tensor l}$.
We reserve the symbols $\Erw[\nix]$, $\pr[\nix]$ for other sources of randomness such as the choice of a random factor graph.
Moreover, for a set $\emptyset\neq U\subset\cX$, $\omega\in\cX$ and $\sigma\in\cX^S$ we let
	\begin{align*}
	\sigma[\omega|U]&=\frac1{|U|}\sum_{u\in U}\vecone\{\sigma(u)=\omega\}.
	\end{align*}
Thus, $\sigma[\nix|U]=(\sigma[\omega|U])_{\omega\in\cX}\in\cP(\cX)$ is the distribution of the spin $\sigma(\vec u)$ for a uniformly random $\vec u\in U$.
Further, for a measure $\mu\in\cP(\cX^S)$ and a sequence $x_1,\ldots,x_l\in S$ of coordinates we let
	$\mu_{x_1,\ldots,x_l}\in\cP(\cX^l)$ be the distribution of the $l$-tuple $(\SIGMA^\mu(x_1),\ldots,\SIGMA^\mu(x_l))$.
In symbols, 
	$\mu_{x_1,\ldots,x_l}(\omega_1,\ldots,\omega_l)=\bck{\vecone\{\forall i\in[l]:\SIGMA^\mu(x_i)=\omega_i\}}_\mu.$

We use the ``regularity lemma'' for discrete probability measures from~\cite{Victor}.
Let us fix a finite set $\cX$ for the rest of this section.
If $\vec V=(V_1,\ldots,V_l)$ is a partition of some set $V$, then we call $\#\vec V=l$ the \emph{size} of $\vec V$.
Moreover, for $\eps>0$ we say that $\mu\in\cP(\cX^n)$ is \emph{$\eps$-regular} on a set $U\subset[n]$ if
for every subset $S\subset U$ of size $|S|\geq\eps|U|$ we have
$$\bck{\TV{\SIGMA[\nix|S]-\SIGMA[\nix|U]}}_{\mu}<\eps.$$
Further, $\mu$ is \emph{$\eps$-regular}  with respect to a partition $\vec V$ if
	there is a set $J\subset[\#\vV]$ such that $\sum_{i\in J}|V_i|\geq(1-\eps)n$ and such that $\mu$ is $\eps$-regular on $V_i$ for all $i\in J$.
Additionally, if $\vec V$ is a partition of $[n]$ and $\vec S=(S_1,\ldots,S_{\#\vS})$ is a partition of $\cX^n$, then
we say that $\mu$ is \emph{$\eps$-homogeneous} w.r.t.\ $(\vec V,\vec S)$ if there is a subset $I\subset[\#\vec S]$ such that the following is true:
\begin{description}
\item[HM1] We have $\mu(S_i)>0$ for all $i\in I$ and $\sum_{i\in I}\mu(S_i)\geq1-\eps$.
\item[HM2] For all $i\in[\#\vec S]$ and $j\in[\#\vec V]$ we have
	$\max_{\sigma,\sigma'\in S_i}\TV{\sigma[\nix|V_j]-\sigma'[\nix|V_j]}<\eps.$
\item[HM3]  For all $i\in I$ the measure $\mu[\nix|S_i]$ is $\eps$-regular with respect to $\vec V$.
\item[HM4] $\mu$ is $\eps$-regular with respect to $\vV$.
\end{description}

\begin{theorem}[{\cite[\Thm~2.1]{Victor}}]\label{Thm_decomp}
For any $\eps>0$ there is an $N=N(\eps,\cX)>0$ such that for every $n>N$, every $\mu\in\cP(\cX^n)$
admits partitions $\vV$ of $[n]$ and $\vS$ of $\cX^n$ with
$\#\vec V+\#\vec S\leq N$ such that
$\mu$ is $\eps$-homogeneous with respect to $(\vec V,\vec S)$.
\end{theorem}

\noindent
A \emph{$(\eps,l)$-state} of $\mu$ is a set $S\subset\cX^n$ such that $\mu(S)>0$ and
	\begin{align*}
	\sum_{x_1,\ldots,x_l\in[n]}\TV{\mu_{x_1,\ldots,x_l}[\nix|S]-\mu_{x_1}[\nix|S]\tensor\cdots\tensor \mu_{x_l}[\nix|S]}<\eps n^l.
	\end{align*}
We call $\mu$ \emph{$(\eps,l)$-symmetric} if the entire cube $\cX^n$ is an $(\eps,l)$-state.

\begin{corollary}[{\cite[\Cor~2.3 and 2.4]{Victor}}]\label{Cor_states}
For any $\eps>0,l\geq3$ there exists $\delta>0$ such that for all $n>1/\delta$ and all $\mu\in\cP(\cX^n)$ the following is true:

If $\mu$ is $(\delta,2)$-symmetric, then $\mu$ is $(\eps,l)$-symmetric.
\end{corollary}

\begin{corollary}[{\cite[\Cor~2.4]{Victor}}]\label{Cor_states2}
For any $\eps>0$ there is a $\gamma >0$ such that for any $\eta>0$ there is $\delta>0$ such that for all $n>1/\delta$, $\mu\in\cP(\cX^n)$ the following is true:

	If $\mu$ is $(\delta,2)$-symmetric, then for any $(\gamma,2)$-state $S$ with $\mu(S)\geq\eta$ we have
	$$\sum_{x\in[n]}\TV{\mu_{x}[\nix|S]-\mu_{x}}<\eps n.$$
\end{corollary}

\begin{lemma}[{\cite[\Lem~2.8]{Victor}}]\label{Lemma_regularSymmetric}
For any $\eps'>0$ there is $\eps>0$ such that for $n>1/\eps$ the following is true:

Assume that $\mu\in\cP(\cX^n)$ is $\eps$-regular with respect to a partition $\vec V$.
The measure $\mu$ is $(\eps',2)$-symmetric if
	\begin{align*}
	\sum_{i\in[\#\vec V]}|V_i|\bck{\TV{\SIGMA[\nix|V_i]-\bck{\TAU[\nix|V_i]}_\mu}}_\mu<\eps n.
	\end{align*}
\end{lemma}

\noindent
Additionally, we need the following observation, whose proof follows that of~\cite[\Cor~2.4]{Victor}.

\begin{lemma}\label{Lemma_switch}
For any $\eps>0$ there is $\xi>0$ and $n_0>0$ such that for any $n>n_0$ and the following holds.
Suppose that $\mu$ is $\xi$-homogeneous w.r.t.\ $(\vV,\vS)$ and that $j\in[\#\vS]$ is such that $\mu[\nix|S_j]$ is $\xi$-regular w.r.t.\ $\vV$.
Then for any $\sigma\in S_j$,
	$$\sum_{i\in[\#\vV]}\sum_{x\in V_i}\TV{\mu_x[\nix|S_j]-\sigma[\nix|V_i]}<\eps n.$$
\end{lemma}
\begin{proof}
Given $\eps>0$ choose $\eta=\eta(\eps)>\xi=\xi(\eta)>0$ sufficiently small and assume that $n$ is large enough.
With $(\vV,\vS)$ and $j$ as above set $\nu=\mu[\nix|S_j]$ for brevity.
Suppose that $i\in[\#\vV]$ is such that $\nu$ is $\xi$-regular on $V_i$ and let $\bar\nu_i(\omega)=\bck{\SIGMA[\omega|V_i]}_\nu$ for $\omega\in\Omega$.
Further, let $W_i(\omega)=\{x\in V_i:\nu_x(\omega)<\bar\nu_i(\omega)-\eta\}$ and suppose $W_i(\omega)\neq\emptyset$.
Then $\bck{\SIGMA[\omega|W_i(\omega)]}_\nu<\bar\nu_i(\omega)-\eta$ by the linearity of expectation.
Hence, by Markov's inequality
	\begin{align*}
	\bck{\vecone\{\SIGMA[\omega|W_i(\omega)]\geq\nu_i(\omega)-\eta/4\}}_\nu&\leq\frac{\bar\nu_i(\omega)-\eta}{\bar\nu_i(\omega)-\eta/4}
		\leq\frac{1-\eta}{1-\eta/4}\leq1-\eta/2.
	\end{align*}
Consequently, {\bf HM2} yields
	$\bck{|\SIGMA[\omega|W_i(\omega)]-\SIGMA[\omega|V_i]|}_\nu\geq
	\bck{|\SIGMA[\omega|W_i(\omega)]-\nu_i(\omega)|}_\nu-\xi\geq\eta^2/8$.
Because $\nu$ is $\xi$-regular on $V_i$, we conclude that $|W_i(\omega)|\leq\xi|V_i|$.
Since this works for every $\omega\in\Omega$, the assertion follows from the triangle inequality and {\bf HM1}--{\bf HM3}.
\end{proof}

\noindent
Finally, we recall the following folklore fact about Poisson random factor graphs.

\begin{fact}\label{Fact_sparse}
For any $\eps>0$ there is $\delta>0$ such that \whp\ the Poisson random factor graph $\G_n$ has the following property.
	\begin{quote}
	For all sets $U\subset\{x_1,\ldots,x_n\}$ of variable nodes of size $|U|\leq\delta n$ we have
		$\sum_{x\in U}d_{\G_n}(x)\leq\eps n$.
	\end{quote}
\end{fact}

\section{Poisson factor graphs}\label{Sec_Poisson}

\noindent{\em Throughout this section we fix $(d,\Omega,k,\Psi,\rho)$ such that (\ref{eqRS}) is satisfied.
	Let $\Psi^*=\{\psi^J:\psi\in\Psi,J\subset[k]\}$.}

\subsection{Proof of \Thm~\ref{Thm_RSBP}}
We begin with the following lemma that will prove useful in \Sec~\ref{Sec_regular} as well.

\begin{lemma}\label{Lemma_cavityRS}
For any integer $L>0$ and any $\alpha>0$ there exist $\eps=\eps(\alpha,L,\Psi)>0$, $n_0=n_0(\eps,L)$ such that the following is true.
Suppose that $G$ is a factor graph with $n>n_0$ variable nodes such that $\psi_a\in\Psi^*$ for all $a\in F(G)$.
Moreover, assume that $\mu_G$ is $(\eps,2)$-symmetric.
If $G^+$ is obtained from $G$ by adding $L$ constraint nodes $b_1,\ldots,b_L$ with weight functions $\psi_{b_1},\ldots,\psi_{b_L}\in\Psi^*$ arbitrarily, then
$\mu_{G^+}$ is $(\alpha,2)$-symmetric and
	\begin{align}\label{eqLemma_cavityRS}
	\sum_{x\in V(G)}\TV{\mu_{G,x}-\mu_{G^+,x}}&<\alpha n.
	\end{align}
\end{lemma}
\begin{proof}
Because all functions $\psi\in\Psi$ are strictly positive, there exists $\delta=\delta(L,\Psi)>0$ such that
for any $\psi_1,\ldots,\psi_L\in\Psi^*$ the following is true.
Suppose that $\psi_i:\Omega^{k_i}\to(0,\infty)$.
Then
	\begin{align}\label{eqLemma_cavityRS0}
	\delta\leq\prod_{i=1}^L\min\{\psi_i(\sigma):\sigma\in\Omega^{k_i}\}\leq\prod_{i=1}^L\max\{\psi_i(\sigma):\sigma\in\Omega^{k_i}\}\leq1/\delta.
	\end{align}
Now, given $\alpha>0$ choose $\eps''=\eps''(\alpha,\delta)>\eps'=\eps'(\eps'')>\eps=\eps(\eps')>0$ small enough,
	let $N=N(\eps)$ be the number promised by \Thm~\ref{Thm_decomp}
	 and assume $n>n_0=n_0(\eps,N)$ is large enough.
By \Thm~\ref{Thm_decomp}  $\mu_{G^+}$ is $\eps$-homogeneous with respect to partitions $(\vV,\vS)$
of $V(G^+)$ and $\Omega^{V(G^+)}$ of sizes $K=\#\vV$ and $L=\#\vS$ such that $K+L\leq N$.
Let  $J$ be the set of all $j\in[L]$ such that $\mu_{G^+}(S_j)\geq\eps/N$ and $\mu_{G^+}[\nix|S_j]$ is $\eps$-regular w.r.t.\ $\vV$.
Then {\bf HM1} and {\bf HM3} ensure that
	\begin{equation}\label{eqLemma_cavityRS2}
	\sum_{j\not\in J}\mu_{G^+}(S_j)<2\eps.
	\end{equation}

We claim that $\mu_{G}[\nix|S_j]$ is $\eps/\delta^2$-regular w.r.t.\ $\vV$ for all $j\in J$.
Indeed, suppose that $\mu_{G^+}$ is $\eps$-regular on $V_i$ and let $U\subset V_i$ be a subset of size $|U|\geq\eps|V_i|$.
Because $G^+$ is obtained from $G$ by adding $L$ constraint nodes, the definition (\ref{eqGibbs}) of the Gibbs measure and the
choice (\ref{eqLemma_cavityRS1}) of $\delta$ ensure that
	\begin{align}\label{eqLemma_cavityRS1}
	\delta\leq\frac{\mu_{G}(\sigma)}{\mu_{G^+}(\sigma)}&\leq1/\delta\qquad\mbox{for all }\sigma\in\Omega^{V(G^+)}.
	\end{align}
Further, (\ref{eqLemma_cavityRS1}) yields
	\begin{align*}
	\bck{\TV{\SIGMA[\nix|V_i]-\SIGMA[\nix|U]}}_{\mu_{G}[\nix|S_j]}
		&=\sum_{\sigma\in\Omega^{V(G)}}\mu_{G}(\sigma|S_j)\TV{\sigma[\nix|V_i]-\sigma[\nix|U]}\leq
			\delta^{-2}\bck{\TV{\SIGMA[\nix|V_i]-\SIGMA[\nix|U]}}_{\mu_{G^+}[\nix|S_j]}<\eps/\delta^2,
	\end{align*}
whence the $\eps/\delta^2$-regularity of $\mu_{G}[\nix|S_j]$ follows.

Moreover, by {\bf HM2} and the triangle inequality for any $j\in J$ we have
	\begin{align}\label{eqLemma_regularSymmetric0}
	\sum_{i\in[\#\vec V]}\frac{|V_i|}{n}\bck{\TV{\SIGMA[\nix|V_i]-\bck{\TAU[\nix|V_i]}_{\mu_{G}[\nix|S_j]}}}_{\mu_{G}[\nix|S_j]}<3\eps.
	\end{align}
In combination with \Lem~\ref{Lemma_regularSymmetric} and the  $\eps/\delta^2$-regularity of $\mu_{G}[\nix|S_j]$, 
(\ref{eqLemma_regularSymmetric0}) implies that
$S_j$ is an $(\eps',2)$-state of $\mu_{G}$ for every $j\in J$, provided that $\eps=\eps(\eps')>0$ was chosen small enough.
In addition, (\ref{eqLemma_cavityRS1}) implies that $\mu_G(S_j)\geq\delta^2\eps/N$ for all $j\in J$.
Consequently, \Cor~\ref{Cor_states2} and our assumption (\ref{eqRS}) entail that for each $j\in J$,
	\begin{align}\label{eqLemma_cavityRS3}
	\sum_{x\in V}\TV{\mu_{G,x}-\mu_{G,x}[\nix|S_j]}&<\eps''n,
	\end{align}
provided $\eps'=\eps'(\eps'')>0$ is sufficiently small and $n>n_0$ is large enough.
Further, by \Lem~\ref{Lemma_switch} and $\eps/\delta^2$-regularity,
	\begin{align*}
	\sum_{i=1}^K\sum_{x\in V_i}\TV{\mu_{G, x}[\nix|S_j]-\sigma[\nix|V_i]}&<\eps''n\qquad\mbox{for all }j\in J,\ \sigma\in S_j.
	\end{align*}
Hence, by (\ref{eqLemma_cavityRS3}) and the triangle inequality,
	\begin{align}\label{eqLemma_cavityRS4}
	\sum_{i=1}^K\sum_{x\in V_i}\TV{\mu_{G, x}-\sigma[\nix|V_i]}&<2\eps''n\qquad\mbox{for all }j\in J,\ \sigma\in S_j.
	\end{align}
Analogously, we obtain from \Lem~\ref{Lemma_switch} that
	\begin{align}\label{eqLemma_cavityRS4a}
	\sum_{i=1}^K\sum_{x\in V_i}\TV{\mu_{G^+, x}[\nix|S_j]-\sigma[\nix|V_i]}&<\eps''n\qquad\mbox{for all }j\in J,\ \sigma\in S_j.
	\end{align}
Combining (\ref{eqLemma_cavityRS4}) and (\ref{eqLemma_cavityRS4a}) and using the triangle inequality, we obtain
	\begin{align}\label{eqLemma_cavityRS5}
	\sum_{x\in V(G)}\TV{\mu_{G, x}-\mu_{G^+, x}[\nix|S_j]}
		&\leq3\eps''n\qquad\mbox{for all }j\in J.
	\end{align}
Moreover, combining (\ref{eqLemma_cavityRS2}) and (\ref{eqLemma_cavityRS5}) and applying the triangle inequality once more, we find
	\begin{align*}
	\sum_{x\in V}\TV{\mu_{G, x}-\mu_{G^+, x}}&\leq
		2\eps n+\sum_{j\in J}\sum_{i=1}^K\sum_{x\in V_i}\mu_{G^+}(S_j)\TV{\mu_{G, x}-\mu_{G^+, x}[\nix|S_j]}<4\eps''n,
	\end{align*}
whence (\ref{eqLemma_cavityRS}) follows.
Finally,  let $\bar\mu_i=\bck{\SIGMA[\nix|V_i]}_{\mu_{G^+}}$.
Then
	\begin{align*}
	\sum_{i=1}^K|V_i|\bck{\TV{\SIGMA[\nix|V_i]-\bar\mu_i}}_{\mu_{G^+}}
		&\leq2\eps n+\sum_{j\in J}\mu_{G^+}(S_j)\sum_{i=1}^K|V_i|\bck{\TV{\SIGMA[\nix|V_i]-\bar\mu_i}}_{\mu_{G^+}[\nix|S_j]}
			&[\mbox{due to~(\ref{eqLemma_cavityRS2})}]\\
		&\leq4\eps n+\sum_{j\in J}\mu_{G^+}(S_j)\sum_{i=1}^K|V_i|\TV{\bck{\SIGMA[\nix|V_i]}_{\mu_{G^+}[\nix|S_j]}-\bar\mu_i}
			&[\mbox{by~{\bf HM2}}]\\
		&\leq4\eps n+\sum_{j\in J}\mu_{G^+}(S_j)\sum_{x\in V(G)}\TV{\mu_{G^+,x}[\nix|S_j]-\bar\mu_i}
			&[\mbox{$\triangle$-inequality}]\\
		&\leq4\eps n+\sum_{j\in J}\mu_{G^+}(S_j)\sum_{x\in V(G)}\TV{\mu_{G^+,x}[\nix|S_j]-\mu_{G,x}}+
			\TV{\bar\mu_i-\mu_{G,x}}\\
		&\leq4\eps'' n+\sum_{j\in J}\mu_{G^+}(S_j)\sum_{x\in V(G)}\TV{\bar\mu_i-\mu_{G,x}}
			&[\mbox{by (\ref{eqLemma_cavityRS5})}]\\
		&\leq7\eps'' n.
		&[\mbox{by ~(\ref{eqLemma_cavityRS2}), (\ref{eqLemma_cavityRS4})}]
	\end{align*}
Thus, {\bf HM4} and \Lem~\ref{Lemma_regularSymmetric} imply that $\mu_{G^+}$ is $(\alpha,2)$-symmetric, provided that $\eps''$
was chosen small enough.
\end{proof}

\noindent
We proceed to prove \Thm~\ref{Thm_RSBP}.
Fix $\eps>0$, choose $L=L(\eps)>0$ and $\gamma=\gamma(\eps,L,\Psi)>\eta=\eta(\gamma)>\delta=\delta(\eta)>0$ small enough and assume
that $n>n_0(\delta)$ is sufficiently large.
Because the distribution of the random factor graph $\G_n$ is symmetric under permutations of the variable nodes,
it suffices to prove that with probability at least $1-\eps$ we have
	\begin{align}\label{eqProof_Thm_RSBP1}
	\sum_{a\in\partial x_n, \sigma\in\Omega}\abs{\mu_{\G_n,x_n\to a}(\sigma)-
			\frac{\prod_{b\in\partial x\setminus a}\mu_{\G_n,b\to x_n}(\sigma)}
				{\sum_{\tau\in\Omega}\prod_{b\in\partial x_n\setminus a}\mu_{\G_n,b\to x_n}(\tau)}}&<\eps\qquad\mbox{and}\\
	\sum_{a\in\partial x_n, \sigma\in\Omega}
				\abs{\mu_{\G_n,a\to x_n}(\sigma)-
			\frac{\sum_{\tau\in\Omega^{\partial a}}\vecone\{\tau(x_n)=\sigma\}\psi_a(\tau)\prod_{y\in\partial a\setminus x_n}
							\mu_{\G_n,y\to a}(\tau(y))}
				{\sum_{\tau\in\Omega^{\partial a}}\psi_a(\tau)\prod_{y\in\partial a\setminus x_n}\mu_{\G_n,y\to a}(\tau(y))}}&<\eps.
					\label{eqProof_Thm_RSBP2}
	\end{align}

To prove (\ref{eqProof_Thm_RSBP1})--(\ref{eqProof_Thm_RSBP2}) we use the following standard trick.
Let $\G'$ be the random factor graph with variable nodes $x_1,\ldots,x_n$ comprising of $m'=\Po(dn(1-1/n)^k/k)$ random constraint nodes
	$a_1,\ldots,a_{m'}$ that do not contain $x_n$.
Moreover, let $\Delta=\Po(dn(1-(1-1/n)^k)/k)$ be independent of $m'$ and obtain $\G''$ from $\G'$ by adding independent random constraint nodes
$b_1,\ldots,b_\Delta$ with $x_n\in\partial b_i$ for all $i\in[\Delta]$.
Then the random factor graph $\G''$ has precisely the same distribution as $\G_n$.
Therefore, it suffices to verify (\ref{eqProof_Thm_RSBP1})--(\ref{eqProof_Thm_RSBP2}) with $\G_n$ replaced by $\G''$.

Since $dn(1-(1-1/n)^k)/k=d+o(1)$, we can choose $L=L(\eps)$ so large that
	\begin{align}					\label{eqProof_Thm_RSBP3}
	\pr\brk{\Delta> L}<\eps/3.
	\end{align}
Furthermore, $\G'$ is distributed precisely as the random factor graph $\G_n$ given that $\partial x_n=\emptyset$.
Therefore, Bayes' rule and our assumption (\ref{eqRS}) imply
	\begin{align}\nonumber
	\pr\brk{\G'\mbox{ fails to be $(\delta,2)$-symmetric}}&\leq
		\pr\brk{\G_n\mbox{ fails to be $(\delta,2)$-symmetric}}/\pr\brk{\partial_{\G_n} x_n=\emptyset}\\
		&\leq\exp(d+o(1))\pr\brk{\G_n\mbox{ fails to be $(\delta,2)$-symmetric}}<\delta,
			\label{eqProof_Thm_RSBP4}
	\end{align}
provided that $n_0$ is chosen large enough.
Combining  (\ref{eqProof_Thm_RSBP4}) and \Cor~\ref{Cor_states}, we see that
	\begin{align}\label{eqProof_Thm_RSBP5}
	\pr\brk{\G'\mbox{ is $(\eta,2+(k-1)L)$-symmetric}|\Delta\leq L}&>1-\delta,
	\end{align}
provided $\delta$ is sufficiently small.

Due to (\ref{eqProof_Thm_RSBP3}) and (\ref{eqProof_Thm_RSBP5}) and the symmetry amongst $b_1,\ldots,b_\Delta$ we just need to prove the following:
	given that $\G'$ is $(\eta,2+(k-1)L)$-symmetric and $0<\Delta\leq L$, with probability at least $1-\eps/L$ we have
	\begin{align}\label{eqProof_Thm_RSBP6}
	\sum_{\sigma\in\Omega}\abs{\mu_{\G'',x_n\to b_1}(\sigma)-
			\frac{\prod_{i=2}^\Delta\mu_{\G'',b_i\to x_n}(\sigma)}
				{\sum_{\tau\in\Omega}\prod_{i=2}^\Delta\mu_{\G'',b_i\to x_n}(\tau)}}&<\eps/L\qquad\mbox{and}\\
	\sum_{\sigma\in\Omega}
			\abs{\mu_{\G'',b_1\to x_n}(\sigma)-
			\frac{\sum_{\tau\in\Omega^{\partial b_1}}\vecone\{\tau(x_n)=\sigma\}\psi_{b_1}(\tau)
					\prod_{y\in\partial b_1\setminus x_n}\mu_{\G_n,y\to b_1}(\tau(y))}
				{\sum_{\tau\in\Omega^{\partial b_1}}\psi_a(\tau)\prod_{y\in\partial b_1\setminus x_n}\mu_{\G_n,y\to b_1}(\tau(y))}}&<\eps/L.
					\label{eqProof_Thm_RSBP7}
	\end{align}
To this end, let $U=\bigcup_{j\geq2}\partial b_j$ be the set of all variable nodes that occur in the constraint nodes $b_2,\ldots,b_\Delta$.
Because $\mu_{\G'',x_n\to b_1}$ is the marginal of $x_n$ in the factor graph $\G''-b_1$,
the definition (\ref{eqGibbs}) of the Gibbs measure entails that for any $\sigma\in\Omega$,
	\begin{align}\nonumber
	\mu_{\G'',x_n\to b_1}(\sigma)&=\frac{\sum_{\tau\in\Omega^{V(\G'')}}
			\vecone\{\tau(x_n)=\sigma\}\prod_{a\in F(\G')}\psi_a(\tau(\partial a))\prod_{j=2}^\Delta\psi_{b_j}(\tau(\partial b_j))
			}{\sum_{\tau\in\Omega^{V(\G'')}}\prod_{a\in F(\G')}\psi_a(\tau(\partial a))\prod_{j=2}^\Delta\psi_{b_j}(\tau(\partial b_j))}
		\\
		&=\frac{\sum_{\tau\in\Omega^U}\vecone\{\tau(x_n)=\sigma\}
			\bck{\vecone\{\forall y\in U\setminus\{x_n\}:\SIGMA(y)=\tau(y)}_{\mu_{\G'}}\prod_{j=2}^\Delta\psi_{b_j}(\tau(\partial b_j))}
			{\sum_{\tau\in\Omega^U}	\bck{\vecone\{\forall y\in U\setminus\{x_n\}:\SIGMA(y)=\tau(y)}_{\mu_{\G'}}
				\prod_{j=2}^\Delta\psi_{b_j}(\tau(\partial b_j))}.
										\label{eqProof_Thm_RSBP8}
	\end{align}
Similarly, because $\mu_{\G'',b_i\to x_n}$ is the marginal of $x_n$ in $\G'+b_i$, we have
	\begin{align}										\label{eqProof_Thm_RSBP8a}
	\mu_{\G'',b_i\to x_n}(\sigma)&=\frac{\sum_{\tau\in\Omega^{\partial b_i}}\vecone\{\tau(x_n)=\sigma\}
			\bck{\vecone\{\forall y\in \partial b_i\setminus\{x_n\}:\SIGMA(y)=\tau(y)}_{\mu_{\G'}}\psi_{b_i}(\tau)}
			{\sum_{\tau\in\Omega^{\partial b_i}}
			\bck{\vecone\{\forall y\in \partial b_i\setminus\{x_n\}:\SIGMA(y)=\tau(y)}_{\mu_{\G'}}\psi_{b_i}(\tau)}.
	\end{align}

To prove (\ref{eqProof_Thm_RSBP6}), recall that the variable nodes $\partial b_j\setminus x_n$ are chosen uniformly and independently for each $j\geq2$.
Therefore, if $\G'$ is $(\eta,(k-1)L)$-symmetric and $0<\Delta\leq L$, then
	\begin{align*}
	\sum_{\tau\in\Omega^U}\Erw\brk{
		\abs{\textstyle\bck{\vecone\{\forall y\in U\setminus\{x_n\}:\SIGMA(y)=\tau(y)}_{\mu_{\G'}}-\prod_{y\in U}\mu_{\G',y}(\tau(y))}\big|\G'}
	\leq2\eta.
	\end{align*}
Hence, by Markov's inequality with probability at least $1-\eta^{1/3}$ we have
	\begin{align}
	\sum_{\tau\in\Omega^U}
		\abs{\textstyle\bck{\vecone\{\forall y\in U\setminus\{x_n\}:\SIGMA(y)=\tau(y)}_{\mu_{\G'}}-\prod_{y\in U}\mu_{\G',y}(\tau(y))}&<\eta^{1/3}.
								\label{eqProof_Thm_RSBP9}
	\end{align}
Set
	\begin{align}\label{eqnuMessages}
	\nu_i(\sigma)&=\sum_{\tau\in\Omega^{\partial b_i}}\vecone\{\tau(x_n)=\sigma\}\psi_{b_i}(\tau)
		\prod_{y\in\partial b_i\setminus x_n}\mu_{\G',y}(\tau(y)).
	\end{align}
\Whp\ for any $1\leq i<j\leq\Delta$ we have $\partial b_i\cap\partial b_j=\{x_n\}$.
Hence, assuming that $\eta=\eta(\gamma)>0$ is chosen small enough, we obtain from (\ref{eqProof_Thm_RSBP8}), 
	(\ref{eqProof_Thm_RSBP8a}), (\ref{eqProof_Thm_RSBP9})  that
with probability at least $1-\gamma$,
	\begin{align}								\label{eqProof_Thm_RSBP10}
	\abs{\mu_{\G'',x_n\to b_1}(\sigma)-\frac{\prod_{i=2}^\Delta\nu_i(\sigma)}{\sum_{\tau\in\Omega}\prod_{i=2}^\Delta\nu_i(\tau)}}<\gamma
		\qquad\mbox{and}\qquad
	\abs{\mu_{\G'',b_i\to x_n}(\sigma)-\frac{\nu_i(\sigma)}{\sum_{\tau\in\Omega}\nu_i(\tau)}}&<\gamma
	\qquad\mbox{for all }i\in[\Delta].
	\end{align}
Hence, (\ref{eqProof_Thm_RSBP6}) follows from (\ref{eqProof_Thm_RSBP10}), provided that $\gamma$ is chosen small enough.

Finally, to prove  (\ref{eqProof_Thm_RSBP7}) we use \Lem~\ref{Lemma_cavityRS}.
Let $\G'''=\G''-b_1$ be the graph obtained from $\G'$ by merely adding $b_2,\ldots,b_\Delta$.
Given that $\G'$ is $(\eta,2)$-symmetric, \Lem~\ref{Lemma_cavityRS} and \Cor~\ref{Cor_states} imply that
 $\G'''$ is $(\gamma^3,k-1)$-symmetric. 
As $\partial b_1\setminus x_n$ is a random subset of size at most $k-1$ chosen independently of $b_2,\ldots,b_\Delta$,
	we conclude that with probability at least $1-\gamma$ over the choice of $\G''$,
	\begin{align}\nonumber
	2\gamma&>\sum_{\tau\in\Omega^{\partial b_1}}\abs{\bck{\vecone\{\forall y\in\partial b_1\setminus x_n:\SIGMA(y)=\tau(y)\}}_{\mu_{\G'''}}
		-\prod_{y\in\partial b_1\setminus x_n}\mu_{\G''',y}(\tau(y))}\\
		&=
		\sum_{\tau\in\Omega^{\partial b_1}}\abs{\bck{\vecone\{\forall y\in\partial b_1\setminus x_n:\SIGMA(y)=\tau(y)\}}_{\mu_{\G'''}}
		-\prod_{y\in\partial b_1\setminus x_n}\mu_{\G'',y\to b_1}(\tau(y))}.
									\label{eqProof_Thm_RSBP11}
	\end{align}
Moreover, (\ref{eqLemma_cavityRS}) implies that with probability at least $1-\gamma$,
	\begin{align}								\label{eqProof_Thm_RSBP12}
	2\gamma&>\sum_{\tau\in\Omega^{\partial b_1}}\abs{\bck{\vecone\{\forall y\in\partial b_1\setminus x_n:\SIGMA(y)=\tau(y)\}}_{\mu_{\G'''}}
		-\prod_{y\in\partial b_1\setminus x_n}\mu_{\G',y}(\tau(y))}.
	\end{align}
Finally,  (\ref{eqProof_Thm_RSBP7}) follows from
	(\ref{eqnuMessages})--(\ref{eqProof_Thm_RSBP12}), provided $\gamma$ is chosen small enough.

\subsection{Proof of \Cor~\ref{Thm_RSBethe}}
Following Aizenman-Sims-Starr~\cite{Aizenman} we are going to show that
	\begin{align}\label{eqAizenman}
	\lim_{n\to\infty}\Erw\brk{\ln\frac{Z_{\G_n}}{Z_{\G_{n-1}}}}&= B.
	\end{align}
The assertion then follows by summing on $n$.
To prove (\ref{eqAizenman}) we will couple the random variables $Z_{\G_{n-1}},Z_{\G_n}$
by way of a third random factor graph $\hat\G$; a similar coupling was used in~\cite{COPS}.
Specifically, let $\hat\G$ be the random factor graph with variable nodes $V(\hat\G)=\{x_1,\ldots,x_n\}$ obtained by including 
	$\hat m=\Po(n\hat d/k)$ independent random constraint nodes,
where
	$$\hat d=d\bcfr n{n-1}^{k-1}.$$
For each constraint node $a$ of $\hat\G$ the weight function $\psi_a$ is chosen from the distribution $\rho$ independently.
Further, set $p=((n-1)/n)^{k-1}$ and let $\G'$ be a random graph obtained from $\hat\G$ by deleting each constraint node
with probability $1-p$ independently.
Let $A$ be the (random) set of constraints removed from $\hat\G$ to obtain $\G'$.
In addition, obtain $\G''$ from $\hat\G$ by selecting a variable node $\vec x$ uniformly at random and removing all constraints $a\in\partial_{\hat\G}\vec x$
along with $\vec x$ itself.
Then $\G'$ is distributed as $\G_n$ and $\G''$ is distributed as $\G_{n-1}$ plus an isolated variable.
Thus,
	\begin{align}\label{eqCoupling}
	Z_{\G_n}&\stacksign{$d$}=Z_{\G'},
		&Z_{\G_{n-1}}&\stacksign{$d$}=Z_{\G''}.
	\end{align}

\begin{fact}\label{Claim_contigPoisson}
The two factor graph distributions $\hat\G,\G_n$ have total variation distance $O(1/n)$.
\end{fact}
\begin{proof}
Given that $|F(\hat\G)|=|F(\G_n)|$ both factor graphs are identically distributed.
Moreover, $|F(\G_n)|$ is Poisson with mean $dn/k$, which has total variation distance $O(1/n)$ from the distribution of $\hat m$.
\end{proof}

\noindent
For $x\in V(\hat\G)$, $b\in F(\hat\G)$ we define
	\begin{align}\label{eqS1}
	S_1(x)&=\ln\brk{\sum_{\sigma\in\Omega}\prod_{a\in\partial_{\hat\G} x}\mu_{\hat\G,a\to x}(\sigma)},&
	S_2(x)&=\sum_{a\in\partial_{\hat\G} x}
				\ln\brk{\sum_{\tau\in\Omega^{\partial a}}\psi_a(\tau)\prod_{y\in\partial a}\mu_{\hat\G,y\to a}(\tau(y))},\\
	S_3(x)&=-\sum_{a\in\partial_{\hat\G} x}\ln\brk{\sum_{\tau\in\Omega}\mu_{\hat\G, x\to a}(\tau)\mu_{\hat\G,a\to x}(\tau)},&
	S_4(b)&=\ln\brk{\sum_{\sigma\in\Omega^{\partial b}}\psi_b(\sigma)\prod_{y\in\partial b}
		\mu_{\hat\G,y\to b}(\sigma(y))}.
			\label{eqS4}
	\end{align}

\begin{lemma}\label{Claim_eqRSBethe1}
\Whp\ we have	$\ln\frac{Z_{\hat\G}}{Z_{\G'}}=o(1)+\sum_{a\in A}S_4(a).$
\end{lemma}
\begin{proof}
Given $\eps>0$ let $L=L(\eps)>0$ be a large enough number, let $\gamma=\gamma(\eps,L,\Psi)>\delta=\delta(\gamma)>0$ be small enough and assume that $n$ is sufficiently large.
Let $X=|A|$.
Then the construction of $\G'$ ensures that 
	\begin{align}\label{eqClaim_eqRSBethe1_1}
	\pr\brk{X>L}<\eps.
	\end{align}

Instead of thinking of $\G'$ as being obtained from $\hat\G$ by removing $X$ random constraints, we can think of $\hat\G$
as being obtained from $\G'$ by adding $X$ independent random constraint nodes $a_1,\ldots,a_X$.
More precisely, let $\G_0'=\G'$ and $\G_i'=\G'_{i-1}+a_i$ for $i\in[X]$.
Then given $X$ the triple $(\G',\hat\G,A)$ has the same distribution as $(\G',\G_X',\{a_1,\ldots,a_X\})$.

Moreover, because $p\hat d n/k=dn/k$, $\G'$ has the same distribution as $\G_n$.
Therefore, our assumption (\ref{eqRS})  implies that $\G'$ is $(o(1),2)$-symmetric \whp\
Hence, \Lem~\ref{Lemma_cavityRS} implies that $\G'_{i-1}$ retains $(o(1),2)$-symmetry \whp\ for any $1\leq i\leq \min\{X,L\}$.
Consequently, \Cor~\ref{Cor_states} implies that
$\G'_{i-1}$ is $(o(1),k)$-symmetric \whp\ 
Since $\partial b_i$ is chosen uniformly and independently of $b_1,\ldots,b_{i-1}$, Markov's inequality thus shows that for every $1\leq i\leq\min\{X,L\}$,
	\begin{align*}
	\pr\brk{\sum_{\tau\in\Omega^k}\abs{\bck{\vecone\{\forall y\in\partial a_i:\SIGMA(y)=\tau(y)\}}_{\mu_{\G'_{i-1}}}-
			\prod_{y\in\partial a_i}\mu_{\G'_{i-1},y}(\tau(y))}\geq\delta }&<\delta,
	\end{align*}
provided $n$ is big enough.
Further, since the constraints $(a_i)_{i\in[X]}$ are chosen independently
and because $\mu_{\hat\G,y\to a_i}(\tau(y))$ is the marginal in the factor graph without $a_i$,
(\ref{eqLemma_cavityRS}) and (\ref{eqClaim_eqRSBethe1_1}) imply that
	\begin{align*}
	\pr\brk{\forall i\in[X]:
		\sum_{\tau\in\Omega^k}\abs{\prod_{y\in\partial a_i}\mu_{\hat\G,y\to a_i}(\tau(y))-
			\prod_{y\in\partial a_i}\mu_{\G'_{i-1},y}(\tau(y))}\geq\delta }&<2\eps.
	\end{align*}
Hence, with probability at least $1-3\eps$ the bound
	\begin{align}\label{eqClaim_eqRSBethe1_2}
	\sum_{\tau\in\Omega^k}\abs{\bck{\vecone\{\forall y\in\partial a_i:\SIGMA(y)=\tau(y)\}}_{\mu_{\G'_{i-1}}}-
			\prod_{y\in\partial a_i}\mu_{\hat\G,y\to a_i}(\tau(y))}<2\delta
	\end{align}
holds for all $i\in[X]$ simultaneously.
Further,  the definition (\ref{eqGibbs}) of the partition function entails that for any $i\in[X]$,
	$$Z_{\G'_i}/Z_{\G'_{i-1}}=\sum_{\sigma\in\Omega^{\partial a_i}}\psi_{a_i}(\sigma)
	\bck{\vecone\{\forall y\in\partial a_i:\SIGMA(y)=\sigma(y)\}}_{\mu_{\G'_{i-1}}}.$$
Thus, if (\ref{eqClaim_eqRSBethe1_2}) holds and if $\delta$ is chosen sufficiently small, then
	\begin{align*}
	\abs{\frac{Z_{\G'_i}}{Z_{\G'_{i-1}}}-\sum_{\sigma\in\Omega^{\partial a_i}}\psi_{a_i}(\sigma)\prod_{y\in\partial a_i}
		\mu_{\hat\G,a_i\to y}(\tau(y))}
		<\gamma.
	\end{align*}
Finally, the assertion follows by taking logarithms and summing over $i=1,\ldots,X$.
\end{proof}

\begin{lemma}\label{Claim_eqRSBethe2a}
Let $U=\bigcup_{a\in\partial_{\hat\G}\vec x}\partial a$.
Then \whp\ we have
	\begin{align*}
	\ln\frac{Z_{\hat\G}}{Z_{\G''}}=o(1)+\ln\sum_{\tau\in\Omega^U}
		\prod_{a\in\partial_{\hat\G}\vec x}\brk{\psi_{a}(\tau(\partial a))\prod_{y\in\partial a\setminus\vec  x}\mu_{\hat\G,y\to a}(\tau(y))}.
	\end{align*}
\end{lemma}
\begin{proof}
Given $\eps>0$ let $L=L(\eps)>0$ be a large enough, let $\gamma=\gamma(\eps,L)>\delta=\delta(\gamma)>0$ be small enough and assume that $n$ is sufficiently large.
Letting $X=|\partial_{\hat\G}\vec x|$, we can pick $L$ large enough so that
	\begin{align}\label{eqClaim_eqRSBethe2_1}
	\pr\brk{X>L}<\eps.
	\end{align}
As in the previous proof, we turn the tables: we think of $\hat\G$  as being obtained from $\G''$ by adding a new variable node $\vec x$ and $X$ independent random constraint nodes $a_1,\ldots,a_X$ such that $x\in\partial a_i$ for all $i$.

The assumption (\ref{eqRS}), \Lem~\ref{Lemma_cavityRS} and \Cor~\ref{Cor_states} imply that
	\begin{align}\label{eqClaim_eqRSBethe2_2}
	\pr\brk{\sum_{\tau\in\Omega^{U\setminus\{\vec x\}}}\abs{\bck{\vecone\{\forall y\in U\setminus\{\vec x\}:\SIGMA(y)=\tau(y)\}}_{\G''}-
			\prod_{i=1}^X\prod_{y\in\partial a_i\setminus\vec  x}\mu_{\hat\G,y\to a_i}(\tau(y))}\geq\delta\bigg|X\leq L}&=o(1).
	\end{align}
Furthermore, (\ref{eqGibbs}) yields
	\begin{align*}
	\frac{Z_{\hat\G}}{Z_{\G''}}&=\sum_{\tau\in\Omega^U}\bck{\vecone\{\forall y\in U\setminus\{\vec x\}:\SIGMA(y)=\tau(y)\}}_{\G''}
		\prod_{i=1}^X\psi_{a_i}(\tau(\partial a_i)).
	\end{align*}
Hence, (\ref{eqClaim_eqRSBethe2_1}) and (\ref{eqClaim_eqRSBethe2_2}) show that with probability at least $1-2\eps$, 
	\begin{align}\label{eqClaim_eqRSBethe2_3}
	\abs{\frac{Z_{\hat\G}}{Z_{\G''}}-\sum_{\tau\in\Omega^U}
		\prod_{i=1}^X\brk{\psi_{a_i}(\tau(\partial a_i))\prod_{y\in\partial a_i\setminus\vec  x}\mu_{\hat\G,y\to a_i}(\tau(y))}
		}&<\gamma.
	\end{align}
The assertion follows by taking logarithms.
\end{proof}

\begin{corollary}\label{Claim_eqRSBethe2}
\Whp\ we have $\ln\frac{Z_{\hat\G}}{Z_{\G''}}=S_1(\vec x)+S_2(\vec x)+S_3(\vec x)+o(1).$
\end{corollary}
\begin{proof}
Let $a_1,\ldots,a_X$ be the constraint nodes adjacent to $\vec x$ and let $U=\bigcup_{i=1}^X\partial_{\hat\G} a_i$.
With probability $1-O(1/n)$ for all $1\leq i<j\leq X$ we have $\partial a_i\cap\partial a_j\setminus\{\vec x\}=\emptyset$.
If so, then
		\begin{align*}
		\sum_{\tau\in\Omega^U}\prod_{i=1}^X\brk{\psi_{a_i}(\tau(\partial a_i))\prod_{y\in\partial a_i\setminus x}\mu_{\hat\G,y\to a_i}(\tau(y))}&=
		\sum_{\tau(x)\in\Omega}\prod_{i=1}^X\brk{\sum_{\tau\in\Omega^{\partial a_i\setminus x}}
			\psi_{a_i}(\tau(\partial a_i))\prod_{y\in\partial a_i\setminus x}\mu_{\hat\G,y\to a_i}(\tau(y))}.
	\end{align*}
Hence, \Lem~\ref{Claim_eqRSBethe2a} entails
	\begin{align}\label{eqClaim_eqRSBethe2_4}
	\pr\brk{\ln\frac{Z_{\hat\G}}{Z_{\G''}}=S+o(1)}&=1-o(1),\qquad\mbox{where}\quad
	S=\ln\sum_{\tau(x)\in\Omega}\prod_{i=1}^X\brk{\sum_{\tau\in\Omega^{\partial a_i\setminus\vec x}}
			\psi_{a_i}(\tau(\partial a_i))\prod_{y\in\partial a_i\setminus\vec x}\mu_{\hat\G,y\to a_i}(\tau(y))}.
	\end{align}
Further, by Fact~\ref{Claim_contigPoisson} and \Thm~\ref{Thm_RSBP} 
the messages $\mu_{\hat\G,\nix\to\nix}$ are a $o(1)$-approximate Belief Propagation fixed point \whp\
Consequently, since $\vec x$ is chosen uniformly, we conclude that \whp
	\begin{align}\label{eqClaim_eqRSBethe2_5}
	S&=o(1)+		\ln\brk{\sum_{\tau\in\Omega}\prod_{i=1}^X\mu_{\hat\G,a_i\to\vec  x}(\tau)}
			+\sum_{i=1}^X\ln\brk{\sum_{\tau\in\Omega^{\partial a_i}}\psi_{a_i}(\tau)	
						\prod_{y\in\partial a_i\setminus\vec x}\mu_{\hat\G,y\to a_i}(\tau(y))}.
	\end{align}
Moreover, again due to the $o(1)$-approximate Belief Propagation fixed point property, \whp\ we have
	\begin{align}\label{eqClaim_eqRSBethe2_6}
	\ln\sum_{\tau\in\Omega}\mu_{\hat\G,\vec x\to a_i}(\tau)\mu_{\hat\G,a_i\to\vec  x}(\tau)
		&=o(1)+\ln\frac{\sum_{\tau\in\Omega^{\partial a_i}}\psi_{a_i}(\tau)\prod_{y\in\partial a_i}\mu_{\hat\G,y\to a_i}(\tau(y))}
			{\sum_{\tau\in\Omega^{y\in\partial a_i}}\psi_{a_i}(\tau)\prod_{y\in\partial a_i\setminus x}\mu_{\hat\G,y\to a_i}(\tau(y))}
			\quad\mbox{ for all $i\in[X]$}.
	\end{align}
Plugging (\ref{eqClaim_eqRSBethe2_6}) into~(\ref{eqClaim_eqRSBethe2_5}), we see that \whp
	\begin{align}\nonumber
	S&=o(1)+
		\ln\brk{\sum_{\sigma\in\Omega}\prod_{i=1}^X\mu_{\hat\G,a_i\to\vec  x}(\sigma)}
			+\sum_{i=1}^X
				\ln\brk{\sum_{\tau\in\Omega^{\partial a_i}}\psi_{a_i}(\tau)\prod_{y\in\partial a_i}\mu_{\hat\G,y\to a_i}(\tau(y))}
					-\ln\brk{\sum_{\tau\in\Omega}\mu_{\hat\G,\vec x\to a_i}(\tau)\mu_{\hat\G,a_i\to\vec  x}(\tau)}\\
		&=S_1(\vec x)+S_2(\vec x)+S_3(\vec x)+o(1).
			\label{eqClaim_eqRSBethe2_7}
	\end{align}
Thus, the assertion follows from (\ref{eqClaim_eqRSBethe2_4}).
\end{proof}

\noindent
Combining \Lem~\ref{Claim_eqRSBethe1} and \Cor~\ref{Claim_eqRSBethe2}, we see that \whp\ $\hat\G$ is such that
	\begin{align*}
	\Erw\brk{\ln\frac{Z_{\G'}}{Z_{\G''}}\bigg|\hat\G}
		&=o(1)+\frac1n\brk{\sum_{x\in V(\hat\G)}(S_1(x)+S_3(x))+\sum_{a\in F(\hat\G)}S_4(a)}.
	\end{align*}
Moreover, by our assumption and Fact~\ref{Claim_contigPoisson} the r.h.s.\ converges to $B$ in probability.
Thus, \Cor~\ref{Thm_RSBethe} follows by taking the expectation over $\hat\G$.

\subsection{Proof of \Cor~\ref{Cor_RSBethe}}

\noindent
We begin by deriving formulas for the variable and constraint marginals in terms of the messages.

\begin{lemma}\label{Claim_VarMargs}
We have
	\begin{align}\label{eqClaim_VarMargs}
	\frac1n\Erw\sum_{i=1}^n\sum_{\sigma\in\Omega}\abs{\mu_{\G_n,x_i}(\sigma)-\frac{\prod_{a\in\partial x_i}\mu_{\G_n,x_i\to a}(\sigma)}
		{\sum_{\tau\in\Omega}\prod_{a\in\partial x_i}\mu_{\G_n,x_i\to a}(\tau)}}=o(1).
	\end{align}
\end{lemma}
\begin{proof}
Proceeding along the lines of the proof of \Thm~\ref{Thm_RSBP}, we let
$\G'$ be the random factor graph on $x_1,\ldots,x_n$ containing $m'=\Po(dn(1-1/n)^k/k)$ random constraint nodes
that do not touch $x_n$.
Obtain $\G''$ from $\G'$ by adding $\Delta=\Po(dn(1-(1-1/n)^k)/k)$ random constraint nodes $b_1,\ldots,b_\Delta$ that contain $x_n$
so that $\G''$ is distributed as $\G_n$.
Let $U=\bigcup_{i=1}^\Delta\partial b_i$.
In complete analogy to (\ref{eqProof_Thm_RSBP8}) we obtain the formula
	\begin{align}\label{eqClaim_VarMargs1}
		\mu_{\G'',x_n}(\sigma)
		&=\frac{\sum_{\tau\in\Omega^U}\vecone\{\tau(x_n)=\sigma\}
			\bck{\vecone\{\forall y\in U\setminus\{x_n\}:\SIGMA(y)=\tau(y)}_{\mu_{\G'}}\prod_{j=1}^\Delta\psi_{b_j}(\tau(\partial b_j))}
			{\sum_{\tau\in\Omega^U}	\bck{\vecone\{\forall y\in U\setminus\{x_n\}:\SIGMA(y)=\tau(y)}_{\mu_{\G'}}
				\prod_{j=1}^\Delta\psi_{b_j}(\tau(\partial b_j))}.
	\end{align}
Further, since $\pr\brk{\partial_{\G_n} x_n=\emptyset}=\Omega(1)$, (\ref{eqRS}) implies that $\G'$ is $(o(1),2)$-symmetric \whp\
Therefore, \Cor~\ref{Cor_states} shows that $\G'$ is in fact $(o(1),2+(k-1)\Delta)$-symmetric \whp\
Consequently, \whp
	\begin{align}\label{eqClaim_VarMargs2}
	\sum_{\tau\in\Omega^U}
		\abs{\bck{\vecone\{\forall y\in U\setminus\{x_n\}:\SIGMA(y)=\tau(y)}_{\mu_{\G'}}-\prod_{y\in U\setminus\{x_n\}}\mu_{\G',y}(\tau(y))}
		=o(1).
	\end{align}
Hence, with $\nu_i(\sigma)$ from (\ref{eqnuMessages}) we see that \whp
	\begin{align}\label{eqClaim_VarMargs3}
	\sum_{i=1}^\Delta\sum_{\sigma\in\Omega}\abs{\mu_{\G'',b_i\to x_n}(\sigma)-\frac{\nu_i(\sigma)}{\sum_{\tau\in\Omega}\nu_i(\tau)}}&=o(1)
	\end{align}
Finally, combining (\ref{eqClaim_VarMargs1})--(\ref{eqClaim_VarMargs3}) completes the proof.
\end{proof}

\begin{lemma}\label{Claim_ConstrMargs}
We have
	\begin{align}\label{eqClaim_ConstrMargs}
	\frac1n\Erw\sum_{a\in F(\G_n)}\sum_{\sigma\in\Omega^{\partial a}}
		\abs{\mu_{\G_n,a}(\sigma)-\frac{\psi_a(\sigma)\prod_{x\in\partial a}\mu_{\G_n,x\to a}(\sigma(x))}
		{\sum_{\tau\in\Omega^{\partial a}}\psi_a(\tau)\prod_{x\in\partial a}\mu_{\G_n,x\to a}(\tau(x))}}=o(1).
	\end{align}
\end{lemma}
\begin{proof}
Obtain $\G'$ from $\G_n$ by adding one single random constraint node $a$.
Then the distribution of the pair $(\G',a)$ is at total variation distance $O(1/n)$ from the distribution of the pair $(\G_n,\vec a)$, where
$\vec a$ is a random constraint node of $\G_n$ given $F(\G_n)\neq\emptyset$.
Therefore, it suffices to prove the estimate
	\begin{align}\label{eqClaim_ConstrMargs1}
	\Erw\sum_{\sigma\in\Omega^{\partial a}}
		\abs{\mu_{\G',a}(\sigma)-\frac{\psi_a(\sigma)\prod_{x\in\partial a}\mu_{\G',x\to a}(\sigma(x))}
		{\sum_{\tau\in\Omega^{\partial a}}\psi_a(\tau)\prod_{x\in\partial a}\mu_{\G',x\to a}(\tau(x))}}=o(1).
	\end{align}
The assumption (\ref{eqRS}) and \Cor~\ref{Cor_states} imply that \whp\ $\mu_{\G_n}$ is $(o(1),k)$-symmetric.
Hence, because $\partial_{\G'} a$ is random, \whp\ we have $|\mu_{\G_n,\partial a}(\sigma)-\prod_{x\in\partial a}\mu_{\G_n,x}(\sigma(x))|=o(1)$
for all $\sigma\in\Omega^{\partial a}$.
Since  $\mu_{\G_n,x}=\mu_{\G',x\to a}$ for all $x\in\partial a$,  this means that \whp\
	\begin{align}\label{eqClaim_ConstrMargs2}
	\sum_{\sigma\in\Omega^{\partial a}}\abs{\mu_{\G_n,\partial a}(\sigma)-\prod_{x\in\partial a}\mu_{\G',x\to a}(\sigma(x))}=o(1)
	\end{align}
Further, by the definition (\ref{eqGibbs}) of the Gibbs measure,
	\begin{align}\label{eqClaim_ConstrMargs3}
	\mu_{\G',a}(\sigma)&=\frac{\psi_a(\sigma)\mu_{\G_n,\partial a}(\sigma)}{\sum_{\tau\in\Omega^{\partial a}}
		\psi_a(\tau)\mu_{\G_n,\partial a}(\tau)}.
	\end{align}
Thus, (\ref{eqClaim_ConstrMargs1}) just follows from (\ref{eqClaim_ConstrMargs2}) and (\ref{eqClaim_ConstrMargs3}).
\end{proof}

\noindent
Essentially, we will prove \Cor~\ref{Cor_RSBethe} by following the steps of the derivation of the corresponding formula for acyclic factor graphs~\cite[\Chap~14]{MM}.
We just need to allow for error terms that come in because the right hand sides of (\ref{Claim_VarMargs}) and (\ref{Claim_ConstrMargs})
are $o(1)$ rather than $0$ (like in the acyclic case).
Specifically, by Lemma~\ref{Claim_ConstrMargs} \whp\ for all but $o(n)$ constraint nodes $a\in F(\G_n)$ we have
	\begin{align*}
		-\sum_{\sigma\in\Omega^{\partial a}}\mu_{\G_n,a}(\sigma)\ln\frac{\mu_{\G_n,a}(\sigma)}{\psi_a(\sigma)}
			&=o(1)-\sum_{\sigma\in\Omega^{\partial a}}\mu_{\G_n,a}(\sigma)\ln\frac{\prod_{x\in\partial a}\mu_{\G_n,x\to a}(\sigma(x))}
				{\sum_{\tau\in\Omega^{\partial a}}\psi_a(\tau)\prod_{x\in\partial a}\mu_{\G_n,x\to a}(\tau(x))}\\
			&=o(1)+\ln\brk{\sum_{\tau\in\Omega^{\partial a}}\psi_a(\tau)\prod_{x\in\partial a}\mu_{\G_n,x\to a}(\sigma(x))}
				-\sum_{x\in\partial a}\sum_{\sigma\in\Omega}\mu_{\G_n,x}(\sigma)\ln\mu_{\G_n,x\to a}(\sigma).
	\end{align*}
Further, by Lemma~\ref{Claim_VarMargs} \whp\ for all but $o(n)$ variable nodes $x\in V(\G_n)$ we have
	\begin{align*}
	-\sum_{\sigma\in\Omega}\mu_{\G_n,x}(\sigma)\ln\mu_{\G_n,x\to a}(\sigma)&=
		o(1)-\sum_{\sigma\in\Omega}\mu_{\G_n,x}(\sigma)\ln\frac{\prod_{b\in\partial x\setminus a}\mu_{\G_n,b\to x}(\sigma)}
			{\sum_{\tau\in\Omega}\prod_{b\in\partial x\setminus a}\mu_{\G_n,b\to x}(\tau)}\\
		&=o(1)+\ln\brk{\sum_{\tau\in\Omega}\prod_{b\in\partial x\setminus a}\mu_{\G_n,b\to x}(\tau)}
		-\sum_{b\in\partial x\setminus a}\sum_{\sigma}\mu_{\G_n,x}(\sigma)\ln\mu_{\G_n,b\to x}(\sigma).
	\end{align*}
Hence, Fact~\ref{Fact_sparse} implies that \whp\ for all but $o(n)$ constraint nodes $a\in F(\G_n)$,
	\begin{align}\nonumber
	-\sum_{\sigma\in\Omega^{\partial a}}\mu_{\G_n,a}(\sigma)\ln\frac{\mu_{\G_n,a}(\sigma)}{\psi_a(\sigma)}&=
		o(1)+\ln\brk{\sum_{\tau\in\Omega^{\partial a}}\psi_a(\tau_a)\prod_{x\in\partial a}\mu_{\G_n,x\to a}(\tau(x))}\\
	&\quad		+\sum_{x\in\partial a}\brk{\ln\brk{\sum_{\tau\in\Omega}\prod_{b\in\partial x\setminus a}\mu_{\G_n,b\to x}(\tau)}
		-\sum_{b\in\partial x\setminus a}\sum_{\sigma\in\Omega}\mu_{\G_n,x}(\sigma)\ln\mu_{\G_n,b\to x}(\sigma)}.
			\label{eqCor_RSBethe1}
	\end{align}
Moreover, once more by Lemma~\ref{Claim_VarMargs} \whp\ all but $o(n)$ variable nodes $x$ satisfy
	\begin{align}		\nonumber
	-\sum_{\sigma\in\Omega}\mu_{\G_n,x}(\sigma)\ln\mu_{\G_n,x}(\sigma)
		&=o(1)-\sum_{\sigma\in\Omega}
			\mu_{\G_n,x}(\sigma)\ln\frac{\prod_{a\in\partial x}\mu_{\G_n,a\to x}(\sigma)}{\sum_{\tau\in\Omega}
				\prod_{a\in\partial x}\mu_{\G_n,a\to x}(\tau)}\\
			&=o(1)+\ln\brk{\sum_{\tau\in\Omega}\prod_{b\in\partial x}\mu_{\G_n,b\to x}(\tau)}-
				\sum_{b\in\partial x}\sum_{\sigma\in\Omega}\mu_{\G_n,x}(\sigma)\ln\mu_{\G_n,b\to x}(\sigma).
						\label{eqCor_RSBethe2}
	\end{align}
Finally, combining (\ref{eqCor_RSBethe1}) and (\ref{eqCor_RSBethe2}), we see that \whp
	\begin{align*}
	\frac1n\cB_{\G_n}'&=o(1)+
		\sum_{x\in V(\G_n)}\ln\brk{\sum_{\tau\in\Omega}\prod_{b\in\partial x}\mu_{\G_n,b\to x}(\tau)}+
		\sum_{a\in F(\G_n)}\ln\brk{\sum_{\tau\in\Omega^{\partial a}}\psi_a(\tau_a)\prod_{x\in\partial a}\mu_{\G_n,x\to a}(\sigma_x)}\\
	&\qquad\qquad		+\sum_{a\in F(\G_n),x\in\partial a}\ln\frac{\sum_{\tau\in\Omega}\prod_{b\in\partial x\setminus a}\mu_{\G_n,b\to x}(\tau)}
			{\sum_{\tau\in\Omega}\prod_{b\in\partial x}\mu_{\G_n,b\to x}(\tau)}=\frac1n\cB_{\G_n}+o(1).
	\end{align*}
Thus, \Cor~\ref{Cor_RSBethe} follows from \Cor~\ref{Thm_RSBethe}.

\section{Regular factor graphs}\label{Sec_regular}

\noindent {\em
In this section we fix
 $d,\Omega,k,\Psi,\rho ,\eps$ such that $\reg^\eps_n=\reg^\eps_{n,\mathrm{reg}}
	(d,\Omega,k,\Psi,\rho)$ satisfies (\ref{eqregRS}).}

\medskip\noindent
We prove \Thm~\ref{Thm_regRSBP} and \Cor~\ref{Thm_regRSBethe} by adapting the proofs of \Thm~\ref{Thm_RSBP} and \Cor~\ref{Thm_RSBethe}
to the regular factor graph model.
In the proofs in \Sec~\ref{Sec_Poisson} we exploited the Poisson nature of the factor graphs to determine the effect of adding or removing
a few constraint and/or variable nodes.
Here the necessary wiggle room is provided by the ``$\eps$-percolation'' of the otherwise rigid $d$-regular model $\reg_n$.
This enables a broadly similar analysis to that of \Sec~\ref{Sec_Poisson}.
However, some of the details are subtle, most notably the coupling required for the Aizenman-Sims-Starr argument in \Sec~\ref{Sec_regRSBethe}.

\subsection{Proof of Theorem~\ref{Thm_regRSBP}}
Fix $\delta=\delta(\eps,\Psi) >\eta=\eta(\gamma)>0$.
Again it suffices to prove that with probability at least $1-\del$ we have
	\begin{align}\label{eqProof_Thm_regBP1}
	\sum_{a\in\partial x_n, \sigma\in\Omega}\abs{\mu_{\reg_n^\eps,x_n\to a}(\sigma)-
			\frac{\prod_{b\in\partial x\setminus a}\mu_{\reg_n^\eps,b\to x_n}(\sigma)}
				{\sum_{\tau\in\Omega}\prod_{b\in\partial x_n\setminus a}\mu_{\reg_n^\eps,b\to x_n}(\tau)}}&<\del\qquad\mbox{and}\\
	\sum_{a\in\partial x_n, \sigma\in\Omega}
				\abs{\mu_{\reg_n^\eps,a\to x_n}(\sigma)-
			\frac{\sum_{\tau\in\Omega^{\partial a}}\vecone\{\tau(x_n)=\sigma\}\psi_a(\tau)\prod_{y\in\partial a\setminus x_n}\mu_{\reg_n^\eps,y\to a}(\tau(y))}
				{\sum_{\tau\in\Omega^{\partial a}}\psi_a(\tau)\prod_{y\in\partial a\setminus x_n}\mu_{\reg_n^\eps,y\to a}(\tau(y))}}&<\del.
					\label{eqProof_Thm_regBP2}
	\end{align}
 Let $\Delta=d_{\reg_n^\eps}(x_n)$.
 Then $0 \le \Delta \le d$, and $\Pr[\Delta =0] = \Omega( \eps^d)$ by {\bf REG2--REG3}.

Let $\reg'$ be the random factor graph obtained from $\reg_n^\eps$ by deleting all constraint nodes $a$ such that $x_n\in\partial a$.
Then the distribution $\reg'$ is at total variation distance $O(1/n)$ from the distribution of $\reg_n^\eps$ given that $\partial x_n=\emptyset$.
Therefore, the assumption (\ref{eqregRS}) and \Cor~\ref{Cor_states} imply
	\begin{align}\nonumber
	\pr\brk{\reg'\mbox{ fails to be $(\eta,dk)$-symmetric }}&=\pr\brk{\reg_n^\eps\mbox{ fails to be $(\eta,dk)$-symmetric}|\Delta=0}+o(1)\\
		&\leq\frac{\pr\brk{\reg_n^\eps\mbox{ fails to be $(\eta,dk)$-symmetric}}}{\pr\brk{\Delta=0}}+o(1)
		=o(1).\label{eqProof_Thm_regBP5}
	\end{align}
Furthermore, by the Chernoff bound (cf.~(\ref{eq:chernoff}))
	\begin{align}\label{eqProof_Thm_regBP5a}
	\pr\brk{\sum_{a\in F(\reg')}d_{\reg'}(a)\leq(1-\eps/2)dn}=1-o(1).
	\end{align}
Hence, we may condition on the event that $\reg'$ is $(\eta,dk)$-symmetric and that $\sum_{a\in F(\reg')}d_{\reg'}(a)\leq(1-\eps/2)dn$.
If so, then  the set $R$ of variable nodes $x$ of $\reg'$ such that $d_{\reg'}(a)<d$ has size at least $|R|\geq\eps n/2$.

Given $\reg'$ and $\Delta$, let $\reg''$ be the factor graph obtained from $\reg'$ by adding $\Delta$
constraint nodes $b_1,\ldots,b_\Delta$ and perform the following  independently for each $i\in[\Delta]$.
Choose $\psi_i$ from $\Psi$ according to $\rho$ and choose $J_i\subset[k]$
by including each $j \in [k]$ with probability $(1-\eps)$  independently, conditioned on the event that each $|J_i| \ge 1$.
Then let $\psi_{b_i} = \psi_i^{J_i}$.
Connect $x_n$ to each $ b_i$ at a uniformly random position in $J_i$. 
Then connect  constraint $b_i$ at the remaining slots to $|J_i|-1$ variable nodes chosen from $R$	
according to the distribution $q(x)= (d- d_{\reg'}(x))/\sum_{y\in R}(d- d_{\reg'}(x))$.
Our conditioning on $\sum_{a\in F(\reg')}d_{\reg'}(a)\leq(1-\eps/2)dn$ ensures that  all variable nodes of $\reg''$ have degree at most $d$ \whp\
Hence, the distribution of $\reg''$ is at total variation distance $o(1)$ of the distribution of $\reg_n^\eps$ given $\Delta$.

As in the Poisson case we just need to prove the following:
with probability at least $1-\del/d$ we have
	\begin{align}\label{eqProof_Thm_regBP6}
	\sum_{\sigma\in\Omega}\abs{\mu_{\reg'',x_n\to b_1}(\sigma)-
			\frac{\prod_{i=2}^\Delta\mu_{\reg'',b_i\to x_n}(\sigma)}
				{\sum_{\tau\in\Omega}\prod_{i=2}^\Delta\mu_{\reg'',b_i\to x_n}(\tau)}}&<\del/d\qquad\mbox{and}\\
	\sum_{\sigma\in\Omega}
			\abs{\mu_{\reg'',b_1\to x_n}(\sigma)-
			\frac{\sum_{(\tau_y)_{y\in\partial b_1}}\vecone\{\tau_{x_n}=\sigma\}\psi_{b_1}(\tau)
					\prod_{y\in\partial b_1\setminus x_n}\mu_{\reg_n,y\to b_1}(\tau_y)}
				{\sum_{(\tau_y)_{y\in\partial b_1}}\psi_a(\tau)\prod_{y\in\partial b_1\setminus x_n}\mu_{\reg_n,y\to b_1}(\tau_y)}}&<\del/d.
					\label{eqProof_Thm_regBP7}
	\end{align}
If we again let $U=\bigcup_{j\geq2}\partial b_j$ be the set of all variable nodes joined to constraints $b_2,\ldots,b_\Delta$, then since $\mu_{\reg'',x_n\to b_1}$ is the marginal of $x_n$ in the factor graph $\reg''-b_1$ and $\mu_{\reg'',b_i\to x_n}$ is the marginal of $x_n$ in $\reg'+b_i$, we obtain the analogous equations to \eqref{eqProof_Thm_RSBP8} and \eqref{eqProof_Thm_RSBP8a}:
	\begin{align}
	\mu_{\reg'',x_n\to b_1}(\sigma)	&=\frac{\sum_{\tau\in\Omega^U}\vecone\{\tau(x_n)=\sigma\}
			\bck{\vecone\{\forall y\in U\setminus\{x_n\}:\SIGMA(y)=\tau(y)}_{\reg'}\prod_{j=2}^\Delta\psi_{b_j}(\tau(\partial b_j))}
			{\sum_{\tau\in\Omega^U}	\bck{\vecone\{\forall y\in U\setminus\{x_n\}:\SIGMA(y)=\tau(y)}_{\mu_{\reg'}}
				\prod_{j=2}^\Delta\psi_{b_j}(\tau(\partial b_j))},
										\label{eqProof_Thm_regBP8}\\
	\mu_{\reg'',b_i\to x_n}(\sigma)&=\frac{\sum_{\tau\in\Omega^{\partial b_i}}\vecone\{\tau(x_n)=\sigma\}
			\bck{\vecone\{\forall y\in \partial b_i\setminus\{x_n\}:\SIGMA(y)=\tau(y)}_{\mu_{\reg'}}\psi_{b_i}(\tau)}
			{\sum_{\tau\in\Omega^{\partial b_i}}
			\bck{\vecone\{\forall y\in \partial b_i\setminus\{x_n\}:\SIGMA(y)=\tau(y)}_{\mu_{\reg'}}\psi_{b_i}(\tau)}.
	\end{align}
Further, given that $\sum_{a\in F(\reg')}d_{\reg'}(a)\leq(1-\eps/2)dn$ the distribution $q$ is such that $1/(d|R|)\leq q(x)\leq1/|R|$.
Hence, $q$ is ``within a factor of $d$'' of being uniform.
In effect, we can choose $\eta>0$ so small that
	our assumption that $\reg'$ is $(\eta,dk)$-symmetric ensures that with probability
at least $1-\eta^{1/3}$ we have
	\begin{align}									\label{eqProof_Thm_regBP666}
	\sum_{\tau\in\Omega^U}\abs{\bck{\vecone\{\forall y\in U\setminus\{x_n\}:\SIGMA(y)=\tau(y)}_{\reg'}-\prod_{y\in U}\mu_{\reg',y}(\tau(y))}
		<\eta^{1/3}.
	\end{align}
Due to (\ref{eqProof_Thm_regBP5}) and (\ref{eqProof_Thm_regBP5a}) we obtain the assertion
	from (\ref{eqProof_Thm_regBP8})--(\ref{eqProof_Thm_regBP666})
by following the proof of \Thm~\ref{Thm_RSBP} verbatim  from (\ref{eqProof_Thm_RSBP9}).

\subsection{Proof of \Cor~\ref{Thm_regRSBethe}}\label{Sec_regRSBethe}

As in the proof of \Cor~\ref{Thm_RSBethe} we couple $\reg^\eps_{n+1}$ and $\reg^\eps_n$   via a common supergraph $\hat\G$ obtained as follows.
Choose $\hat m$ from the distribution $d+\Po(d(n+1)/k)$ conditional on the event that $k\hat m<dn$.
Then, choose $\hat\G$ with variable nodes $x_1,\ldots,x_{n+1}$
and constraint nodes $\hat a_1,\ldots,\hat a_{\hat m}$ from the distribution $\reg^\eps_{n+1}$ given that $|F(\reg^\eps_{n+1})|=\hat m$.

\begin{claim}
\label{claim:hatgconting}
$\hat\G$ and $\reg_{n+1}^\eps$ are mutually contiguous.
\end{claim}
\begin{proof}
Construct a copy of $\reg_{n+1}^\eps$ by generating $m= \Po(d(n+1)/k)$.  Conditioned on $\hat m =m$, the distributions of $\hat \G$ and $\reg_{n+1}^\eps$ are identical, and so the claim follows from the contiguity of the two Poisson variables, $m$ and $\hat m$. 
\end{proof}

\noindent
Obtain $\G'$ from $\hat\G$ by removing $d$ random constraint nodes.

\begin{claim}
\label{claim:reggp}
$\G'$ is distributed as $\reg^\eps_{n+1}$, up to total variation distance $\exp(-\Omega(\eps^2 n))$.
\end{claim}
\begin{proof}
Couple the distributions as follows: Let $m=\Po(d(n+1)/k)$.  Choose $m$ constraints with independent random weight functions from $\Psi $ according to $\rho$, and choose a set of active slots $J$ including each slot with probability $1-\eps$. Randomly attach the active slots of all constraints to the $n+1$ variable nodes uniformly at random conditioned on no variable node having degree more than $d$.  This construction yields $\reg^\eps_{n+1}$ on the event $\mathcal A$ that the total number of active slots is at most $ dn$. Now add $d$ additional random constraint nodes, with random sets of active slots as above, and attach to  variable nodes at random proportion to the deficit of their degrees from $d$.  On the event $\mathcal A$, this yields the distribution $\hat\G$.  Now remove $d$ constraints at random: the constraints remaining are still matched to uniformly random variable nodes, and so the distribution is  that of $\G'$.  This coupling succeeds if $\mathcal A$ holds, and from a similar Chernoff bound to \eqref{eq:chernoff}, $\Pr[ \mathcal A] \ge 1- e^{-\Omega(\eps^2 n)}$. 
\end{proof}

\noindent
Furthermore, obtain $\G''$ from $\hat \G$ as follows.
\begin{itemize}
\item Select a random variable node $\vec x$ of $\hat \G$.
\item Remove $\vec x$ and all constraint nodes adjacent to $\vec x$.
\item Remove $d-d_{\hat\G}(x)$ further random constraint nodes.
\item Remove each remaining constraint node with probability $1/(n+1)$ independently.
\end{itemize}

\begin{claim}
\label{claim:reggpp}
 $ |\Erw[\ln Z_{\G''}]-\Erw[\ln Z_{\reg_n^\eps}]|=O( \eps)$.
\end{claim}
\begin{proof}
It is not the case that $\G''$ is distributed exactly as $\reg_n^\eps$: the clauses adjacent to $\vec x$ have a different degree distribution than clauses drawn uniformly from $\hat\G$ (for instances, none of them have degree $0$).  Nevertheless, we will show that the two distributions are close enough that we can use $\G''$ in the Aizenman-Sims-Starr scheme.
We will construct the two factor graphs $\vH,\vH''$ with variable nodes $\{x_1,\ldots,x_n\}$ on the same probability space simultaneously
such that the following properties hold:
 \begin{enumerate}
\item Up to total variation distance $\exp(-\Omega(\eps^2 n))$, $\vH''$ is distributed as $\G''$ and $\vH$ is distributed as $\G_n^\eps$.
\item With probability $1-O(\eps)$ the factor graphs $\vH,\vH''$ are identical.
\item The probability that $\vH,\vH''$ differ on more than $2d$ constraint nodes is $\exp(-\Omega(\eps^2 n))$.
\end{enumerate}
Because the set $\Psi$ of possible weight functions is fixed and all $\psi\in\Psi$ are strictly positive,
we have $\ln Z_{\vH},\ln Z_{\vH''}=O(n)$ with certainty.
For the same reason adding or removing a single constraint can only alter $\ln Z_{\vH},\ln Z_{\vH''}$ by some constant $C$.
Therefore, the assertion is immediate from (i)--(iii).

To construct the coupling, we will first couple the degree sequences of the constraints of $\vH,\vH''$
in such a way that with probability $1-O(\eps)$ the sequences are identical and otherwise they differ in at most $2d$ places.
Formally, let $\hat m=d+\Po(d(n+1)/k)$ and let $\hat k=(\hat k_1,\ldots,\hat k_{\hat m})\in\{0,1,\ldots,k\}^{\hat m}$
be a vector with the same distribution as the vector $(d_{\hat\G}(\hat a_1),\ldots,d_{\hat\G}(\hat a_{\hat m}))$ of constraint degrees of $\hat\G$.
Then (\ref{eq:chernoff}) implies that $\hat k$ is distributed as a sequence of independent $\Bin(k,1-\eps)$ variables, up to 
total variation distance $\exp(-\Omega(\eps^2n))$.
Further, let $X''=( X_i'')_{i=0,1,\ldots,k}$ be distributed as the statistics of the degrees of the $d$
constraint nodes deleted from $\hat\G$ in the above construction of $\G''$ given that
that $d_{\hat\G}(\hat a_j)=\hat k_j$ for all $j$;
that is, $X_i''$ is the number of deleted constraint nodes of degree $i$.
Similarly, let $X=(X_i)_{i=0,1,\ldots,k}$ be the statistics of
$d$ elements of the sequence $\hat k$ chosen uniformly without replacement.

Let $\mathcal A$ be the event that 
	\begin{equation}\label{eqEventcA}
	|\hat m-dn/k|\leq\eps/(dk)\quad\mbox{ and }\quad dn(1-2\eps)\leq\sum_{i=1}^{\hat m}\hat k_i\leq dn.
	\end{equation}
Then by {\bf REG2} and the Chernoff bound we have $\pr\brk{\cA}\ge1- \exp(-\Omega(\eps^2 n))$.
To couple $\vH,\vH''$ on the event $\cA$ we make the following two observations.
	\begin{itemize}
	\item $\pr\brk{X_k=d|\cA}=1-O(\eps)$; this is immediate from (\ref{eqEventcA}).
	\item $\pr\brk{X_k''=d|\cA}=1-O(\eps)$; for (\ref{eqEventcA}) implies that the total number
		of variable nodes adjacent to a constraint node of degree less than $d$ is bounded by $3\eps k n$.
	\end{itemize}
Consequently, on $\cA$ we can couple $X,X''$ such that $\pr[X\neq X'']=O(\eps)$.

If $X=X''$, then we choose $\cD=\cD''\subset[\hat m]$ uniformly at random subject to the condition that
$\sum_{i\in\cD}\vecone\{\hat k_i=j\}=X_j$ for all $j=0,1,\ldots,k$.
Otherwise we choose two independent random sets $\cD,\cD''\subset[\hat m]$ with 
$\sum_{i\in\cD}\vecone\{\hat k_i=j\}=X_j$ and $\sum_{i\in\cD''}\vecone\{\hat k_i=j\}=X_j''$ for all $j$.
Further, with $(\xi_i)_{i\geq1}$ a sequence of $\Be(1/(n+1))$ random variables that are mutually
independent and independent of everything else let
	$$\cE=\{i\in[\hat m]\setminus\cD:\xi_i=1\},\qquad\cE''=\{i\in[\hat m]\setminus\cD'':\xi_i=1\}.$$

Now, obtain the random factor graphs $\vH,\vH''$ as follows.
For $i\in\cE\setminus\cD''$ generate constraint nodes $a_i$ of degree $k_i$ by choosing $\partial_{\vH} a_i=\partial_{\vH''} a_i\subset\{x_1,\ldots,x_n\}$ 
uniformly subject to the condition that all variable degrees remain bounded by $d$.
Subsequently, complete $\vH,\vH''$ independently by choosing $\partial_{\vH}a_i$ for $i\in\cD''\setminus\cD$ and $i\in\cD\setminus\cD''$, respectively,
 conditional on no variable degree exceeding $d$.
Moreover, the weight functions are chosen from the distribution $\rho$ so as to coincide in $\vH,\vH''$ for all $i\in\cE\setminus\cD''$.
Finally, if the event $\cA$ does not occur then we choose $\vH,\vH''$ arbitrarily.

It is immediate from the construction and the fact that $\pr\brk{\cA}\ge1- \exp(-\Omega(\eps^2 n))$ that $\vH,\vH''$ satisfy (1) above.
Furthermore, (2) holds because $\vH=\vH''$ if $X=X''$, which occurs with probability $1-O(\eps)$.
In addition, if $X\neq X''$ and $\cA$ occurs, then $\vH,\vH''$ only differ on the constraints in $|\cD\cup\cD''|\leq2d$ constraint nodes, whence (3) follows.
\end{proof}

\noindent
From Claims \ref{claim:reggp} and \ref{claim:reggpp} it follows that 
\begin{align}\label{eqZreg}
\Erw\brk{\ln\frac{Z_{\reg^\eps_{n+1}}}{Z_{\reg^\eps_{n}}}}  &= \Erw\brk{\ln\frac{Z_{\G'}}{Z_{\G''}}} +O(\eps) .
\end{align}

\noindent
Let us define $S_1(x),S_2(x),S_3(x),S_4(a)$ exactly as in (\ref{eqS1})--(\ref{eqS4}) (with the current $\hat\G$).

\begin{lemma}\label{Claim_regG'}
Let $A'=F(\hat\G)\setminus F(\G')$ be the set of constraint nodes of $\hat\G$ that were deleted to obtain $\G'$.
Then \whp
	\begin{align*}
	\ln\frac{Z_{\hat\G}}{Z_{\G'}}&=o(1)+\sum_{a\in A'}S_4(a).
	\end{align*}
\end{lemma}
\begin{proof}
We mimic the proof of \Lem~\ref{Claim_eqRSBethe1}.
Let $\eta=\eta(\eps)>\delta=\delta(\eta)>0$ be small enough and assume that $n>n_0(\delta)$ is sufficiently large.
Instead of thinking of $\G'$ as being obtained from $\hat\G$ by removing $d$ random constraints, we can think of $\hat\G$
as being obtained from $\G'$ by adding $d$ random constraint nodes $a_1,\ldots,a_d$.
More precisely, let $\G_0'=\G'$ and $\G_i'=\G'_{i-1}+a_i$ for $i\in[d]$,
where $\psi_{a_i}$ is chosen according to {\bf REG2} and $\partial a_i$ is chosen uniformly at random
subject to the condition that no variable ends up with degree greater than $d$.
Then we can identify $\hat\G$ with $\G_d'$.
\Whp\ the random factor graph $\G'$ contains at least $\eta n$ variable nodes $x$ such that $d_{\G'}(x)<d$.
Therefore, Claim~\ref{claim:reggp}, assumption (\ref{eqRS}), \Lem~\ref{Lemma_cavityRS} and \Cor~\ref{Cor_states} imply that for every $i\in[d]$,
	\begin{align*}
	\pr\brk{\sum_{\tau\in\Omega^k}\abs{\bck{\vecone\{\forall y\in\partial a_i:\SIGMA(y)=\tau(y)\}}_{\mu_{\G'_{i-1}}}-
			\prod_{y\in\partial a_i}\mu_{\G'_{i-1},y}(\tau(y))}\geq\delta }&=o(1).
	\end{align*}
In addition, (\ref{eqLemma_cavityRS}) yields
	\begin{align*}
	\pr\brk{\forall i\in[d]:
		\sum_{\tau\in\Omega^k}\abs{\prod_{y\in\partial a_i}\mu_{\hat\G,y\to a_i}(\tau(y))-
			\prod_{y\in\partial a_i}\mu_{\G'_{i-1},y}(\tau(y))}\geq\delta }&=o(1).
	\end{align*}
Hence, \whp\ for all $i\in[d]$ simultaneously,
	\begin{align}\label{eqClaim_Claim_regG'_2}
	\sum_{\tau\in\Omega^k}\abs{\bck{\vecone\{\forall y\in\partial a_i:\SIGMA(y)=\tau(y)\}}_{\mu_{\G'_{i-1}}}-
			\prod_{y\in\partial a_i}\mu_{\hat\G,a_i\to y}(\tau(y))}<\delta
	\end{align}
As 
	$Z_{\G'_i}/Z_{\G'_{i-1}}=\sum_{\sigma\in\Omega^{\partial a_i}}\psi_{a_i}(\sigma)
	\bck{\vecone\{\forall y\in\partial a_i:\SIGMA(y)=\sigma(y)\}}_{\mu_{\G'_{i-1}}}$
for all $i\in[d]$,  (\ref{eqClaim_Claim_regG'_2}) implies that \whp
	\begin{align*}
	\abs{\frac{Z_{\G'_i}}{Z_{\G'_{i-1}}}-
		\sum_{\sigma\in\Omega^{\partial a_i}}\psi_{a_i}(\sigma)\prod_{y\in\partial a_i}\mu_{\hat\G,a_i\to y}(\sigma(y))}
		<\eta.
	\end{align*}
The assertion follows by taking logarithms and summing.
\end{proof}

\begin{lemma}\label{Claim_regG''}
Let $A''=F(\hat\G)\setminus(F(\G'')\cup\partial_{\hat\G}\vec x)$.
Then \whp\
	\begin{align*}
	\ln\frac{Z_{\hat\G}}{Z_{\G''}}&=S_1(\vec x)+S_2(\vec x)+S_3(\vec x)+\sum_{a\in A''}S_3(a)+o(1),
	\end{align*}
\end{lemma}
\begin{proof}
Given $\del>0$, let $\gamma=\gamma(\eps,\del)>\eta=\eta(\gamma)>0$ be small enough and assume that $n>n_0(\gamma)$ is sufficiently large.
We can think of $\hat\G$ 
as being obtained from $\G''$ by adding a new variable node $\vec x$, $X\leq d$ random constraint nodes $a_1,\ldots,a_X$ such that $x\in\partial a_i$ for all $i$
and another $Y$ random constraint nodes $a_{X+1},\ldots,a_{X+Y}$ such that $\vec x\not\in\partial a_i$ for $i>X$.
Let $U=\bigcup_{i\leq X+Y}\partial a_i$.
Since $\G''$ has at least $\gamma n$ variables of degree less than $d$ \whp,
Claim~\ref{claim:reggpp}, (\ref{eqRS}), \Lem~\ref{Lemma_cavityRS} and \Cor~\ref{Cor_states} imply that
	\begin{align}\label{eqClaim_regG''_2}
	\pr\brk{\sum_{\tau\in\Omega^{U\setminus\{\vec x\}}}\abs{\bck{\vecone\{\forall y\in U\setminus\{\vec x\}:
		\SIGMA(y)=\tau(y)\}}_{\mu_{\G''}}-
			\prod_{i=1}^X\prod_{y\in\partial a_i\setminus\vec  x}\mu_{\hat\G,y\to a_i}(\tau(y))}\geq\eta}&=o(1).
	\end{align}
As it is immediate from (\ref{eqGibbs}) that
	\begin{align*}
	\frac{Z_{\hat\G}}{Z_{\G''}}&=\sum_{\tau\in\Omega^U}\bck{\vecone\{\forall y\in U\setminus\{x\}:\SIGMA(y)=\tau(y)\}}_{\G''}
		\prod_{i=1}^{X+Y}\psi_{a_i}(\tau(\partial a_i)),
	\end{align*}
(\ref{eqClaim_regG''_2}) shows that \whp
	\begin{align}\label{eqClaim_regG''_3}
	\abs{\frac{Z_{\hat\G}}{Z_{\G''}}-\sum_{\tau\in\Omega^U}
		\prod_{i=1}^{X+Y}\brk{\psi_{a_i}(\tau(\partial a_i))\prod_{y\in\partial a_i\setminus\vec x}\mu_{\hat\G,y\to a_i}(\tau(y))}
		}&<\gamma.
	\end{align}
To complete the proof, we observe that
		\begin{align*}
		\sum_{\tau\in\Omega^U}\prod_{i=1}^{X+Y}\brk{\psi_{a_i}(\tau(\partial a_i))\prod_{y\in\partial a_i\setminus\vec x}
			\mu_{\hat\G,y\to a_i}(\tau(y))}&=
			\brk{\prod_{i=X+1}^{X+Y}\exp(S_4(a_i))}\\
	&\qquad\cdot	\sum_{\tau(\vec x)\in\Omega}\prod_{i=1}^X
		\brk{\sum_{\tau\in\Omega^{\partial a_i\setminus\vec x}}\psi_{a_i}(\tau(\partial a_i))\prod_{y\in\partial a_i\setminus\vec x}
			\mu_{\hat\G,y\to a_i}(\tau(y))}.
	\end{align*}
Hence, plugging this equation into (\ref{eqClaim_regG''_3}) and taking logarithms, we obtain
	\begin{align}\label{eqClaim_regG''_4}
	\pr\brk{\abs{\ln\frac{Z_{\hat\G}}{Z_{\G''}}-S-\sum_{i=X+1}^YS_4(a_i)}<\delta_1}&>1-2\delta_1,\qquad\mbox{where}\\
	S&=\ln\sum_{\tau(x)\in\Omega}\prod_{i=1}^X\brk{\sum_{\tau\in\Omega^{\partial a_i\setminus x}}
			\psi_{a_i}(\tau(\partial a_i))\prod_{y\in\partial a_i\setminus x}\mu_{\hat\G,y\to a_i}(\tau(y))}.\nonumber
	\end{align}
Finally, by Claim~\ref{claim:hatgconting} and \Thm~\ref{Thm_regRSBP}
the messages $(\mu_{\hat\G,\nix\to \nix})$ are an $o(1)$-approximate Belief Propagation fixed point \whp\
Therefore, the calculations (\ref{eqClaim_eqRSBethe2_6})--(\ref{eqClaim_eqRSBethe2_7}) go through and show that
 $|S-(S_1(\vec x)+S_2(\vec x)+S_3(\vec x))|<\del_1$ \whp
\end{proof}

\begin{lemma}\label{Claim_final}
\Whp\ we have 
	\begin{equation}\label{eqfinal}
	\Erw\brk{\ln\frac{Z_{\G'}}{Z_{\G''}}\bigg|\hat\G}=(n+1)^{-1}\cB_{\hat\G}+O(\eps).
	\end{equation}
\end{lemma}
\begin{proof}
Let $\hat m$ be the number of constraint nodes of $\hat\G$.
Combining Claims~\ref{Claim_regG'} and~\ref{Claim_regG''}, we obtain
	\begin{align}\label{eqFinal1}
	\Erw\brk{\ln\frac{Z_{\G'}}{Z_{\G''}}\bigg|\hat\G}&=o(1)+\sum_{i=1}^{n+1}\frac{S_1(x_i)+S_2(x_i)+S_3(x_i)}{n+1}
		+\brk{-d+\frac{\hat m-d}{n+1}+\frac1{n+1}\sum_{i=1}^{n+1}(d-d_{\hat\G}(x_i))}
			\sum_{a\in F(\hat\G)}\frac{S_4(a)}{\hat m}.
	\end{align}
Since $\hat m=\Po(d(n+1)/k)+d$, \whp\ we have
	\begin{align}\label{eqFinal2}
	\hat m^{-1}\bc{-d+\frac{\hat m-d}{n+1}+\frac1{n+1}\sum_x(d-d_{\hat\G}(x))}&=
		 -\frac{kd(1-1/k)}{d(n+1)}+\frac{d\eps(1-\eps)}{\hat m}=-\frac{k-1}{n+1}+O(\eps).
	\end{align}
On the other hand, 
in the sum $(n+1)^{-1}\sum_{i=1}^{n+1}S_2(x_i)$
all but an $O(\eps)$-fraction of the constraint nodes appear $k$ times.
Thus, \whp
	\begin{align}\label{eqFinal3}
	\frac1{n+1}\sum_{i=1}^{n+1}S_2(x_i)&=O(\eps)+\frac k{n+1}\sum_{a\in F(\hat\G)}S_4(a).
	\end{align}
Finally, plugging (\ref{eqFinal2}) and (\ref{eqFinal3}) into (\ref{eqFinal1}),  we obtain (\ref{eqfinal}).
\end{proof}

\noindent
To complete the proof of \Cor~\ref{Thm_regRSBethe} we take $\eps\to 0$ slowly.
We begin with the following observation.

\begin{claim}\label{Claim_Lipshitz}
We have $\frac1n\Erw[\ln Z_{\reg_n}] = \frac{1}{n} \Erw[\ln Z_{\reg_n^\eps}] + O(\eps).$
\end{claim}
\begin{proof}
We recall the following Lipschitz property, which is immediate from (\ref{eqGibbs}): if a factor graph $G'$ is obtained from another factor graph $G$
by adding or removing a single constraint node, then $|\ln Z_G-\ln Z_{G'}|\leq C$ for some fixed number $C=C(\Psi)$. We can couple $\reg_n$ and $\reg_n^\eps$ by forming $G_0$ by choosing $m^\prime= \Po((1-\eps)^k dn/k)$ random constraints, joined at random to variable nodes so that no variable node has degree more than $d$.  To form $\reg_n$ from $G_0$ we add $\lfloor dn/k \rfloor - m^\prime$ additional random constraints; with probability $1-e^{-\Omega(\eps^2 n)}$ the number of additional constraints is $O(\eps n)$. To form $\reg_n^\eps$ from $G_n$, we add $\Po( \binom {k}{j} (1-\eps)^j \eps^{k-j} dn/k)$ random constraints with degree $j$, for $j=1, \dots k-1$.  Again with probability $1 - e^{-\Omega(\eps^2 n)}$ the total number of additional constraints is $O( \eps n)$.  Applying the Lipschitz property twice gives the claim.
\end{proof}

\begin{proof}[Proof of \Thm~\ref{Thm_regRSBethe}]
Let $X=|F(\G')\triangle F(\G'')|$ be the number of constraint nodes in which $\G',\G''$ differ.
As in the previous proof, we know deterministically that $|\ln Z_{\G'}-\ln Z_{\G''}|\leq CX$.
Moreover, the construction of $\G',\G''$ ensures that $X$ has a bounded mean.
Therefore, Markov's inequality and Lemma~\ref{Claim_final} ensure that
	$\Erw\ln(Z_{\G'}/Z_{\G''})=(n+1)^{-1}\Erw[\cB_{\hat\G}]+O(\eps)$.
Hence, by (\ref{eqZreg}), Claim~\ref{claim:hatgconting} and because $\ln Z_{\reg_n^\eps}=O(n)$ with certainty,
	\begin{equation}\label{eqProofThm_regRSBethe}
	\Erw\ln\bc{Z_{\reg^\eps_{n+1}}/Z_{\reg^\eps_{n}}}=(n+1)^{-1}\Erw[\cB_{\hat\G}]+O(\eps)=
		(n+1)^{-1}\Erw[\cB_{\reg^\eps_{n+1}}]+O(\eps).
	\end{equation}
Finally, combining (\ref{eqProofThm_regRSBethe}) with  Lemmas~\ref{Claim_final} and \ref{Claim_Lipshitz}, we obtain
	\begin{align*}
	\lim_{n \to \infty} \frac1n\Erw[\ln Z_{\reg}] &=  \lim_{\eps\searrow0} \lim_{n \to \infty} \frac{1}{n} \Erw[\ln Z_{\reg^\eps}] 
	= \lim_{\eps\searrow0} \lim_{n \to \infty} \Erw\ln(Z_{\reg^\eps_{n+1}}/Z_{\reg^\eps_{n}}) 
	 = B,
	\end{align*}	
as desired.
\end{proof}

\section{Non-reconstruction}\label{Sec_nonre}

\begin{proof}[Proof of \Lem~\ref{prop:nonreconstruction}]
 Let $\G_n$ be distributed according to either $\G_n(d,\Omega,k,\Psi,\rho)$ or $\reg_{n,\mathrm{reg}}^\eps(d,\Omega,k,\Psi,\rho)$.
 Assume that (\ref{eqNonReconstruction}) holds but (\ref{eq:rsnonrecon}) does not. Then there exist $\omega_1,\omega_2\in\Omega$ and $0<\del<\frac{1}{10}$ such that for infinitely many $n$ we have
	\begin{align}\label{eqThm_nonre2}
	\pr\brk{\abs{\bck{\vecone\{\SIGMA\bc{x_1}=\omega_1\}|\SIGMA\bc{x_2}=\omega_2}_{\mu_{\G_n}}-
		\bck{\vecone\{\SIGMA\bc{x_1}=\omega_1\}}_{\mu_{\G_n}}
		}>2\del,\ \bck{\vecone\{\SIGMA\bc{x_2}=\omega_2\}}_{\mu_{\G_n}}>2\del}>2\del.
	\end{align}
Let $\ell$ be a large enough integer and let $\cE$ be the event that the distance between $x_1,x_2$ in $\G_n$ is greater than $2\ell$.
The distribution of $\G_n$ is symmetric with respect to the variables, and the factor graph is sparse: the expected number of variables nodes within distance $2 \ell$ of $x_1$ is $O((kd)^\ell)$, constant with respect to $n$, and so we have $\pr\brk{\cE}=1-o(1)$ as $n\to\infty$.
Therefore, (\ref{eqThm_nonre2}) implies
	\begin{align}\label{eqThm_nonre3}
	\pr\brk{\abs{\bck{\vecone\{\SIGMA\bc{x_1}=\omega_1\}|\SIGMA\bc{x_2}=\omega_2}_{\mu_{\G_n}}-
		\bck{\vecone\{\SIGMA\bc{x_1}=\omega_1\}}_{\mu_{\G_n}}}>\del,\ 
		\bck{\vecone\{\SIGMA\bc{x_2}=\omega_2\}}_{\mu_{\G_n}}>\del,
		\ \cE}>\del.
	\end{align}
To complete the proof, let $\cS$ be the set of all $\sigma\in\Omega^{V(\G_n)}$ such that $\sigma\bc{x_2}=\omega_2$.
If the event $\cE$ occurs, then given $\nabla_\ell(\G_n,x_1,\sigma)$ the value assigned to $x_2$ is fixed for all $\sigma\in\cS$.
Therefore, (\ref{eqThm_nonre3}) implies
	\begin{align*}
	\Erw\bck{\TV{\mu_{\G_n, x_1}-\mu_{\G_n, x_1}\brk{\nix|\nabla_\ell(\G_n,x_1,\SIGMA)}}}_{\mu_{\G_n}}
		&\geq\Erw\brk{\vecone\{\cE\}\bck{\TV{\mu_{\G_n, x_1}-\mu_{\G_n, x_1}\brk{\nix|\nabla_\ell(\G_n,x_1),\SIGMA}}|\cS}_{\mu_{\G_n}}
			\hspace{-2mm}\bck{\vecone\{\cS\}}_{\mu_{\G_n}}}\geq\del^3,
	\end{align*}
in contradiction to (\ref{eqNonReconstruction}).
\end{proof}

\end{document}